\documentclass[11pt,twoside,epsf]{article}
\usepackage{amsmath,amssymb,amsthm}
\usepackage{bm}
\usepackage{amssymb}
\usepackage{color}
\usepackage{amsthm}
\usepackage{galois}
\usepackage{subfigure}
\usepackage{graphicx}
\usepackage{epstopdf}
\newfont{\bb}{msbm10}

\makeatletter 
\renewcommand\theequation{\thesection.\arabic{equation}}
\@addtoreset{equation}{section}
\makeatother 

\begin{document}
\cleardoublepage
\pagestyle{myheadings}

\bibliographystyle{plain}

\title{Poisson Integrators based on splitting method for Poisson systems}
\author
{Beibei Zhu \\
{\it Department of Applied Mathematics, School of Mathematics and Physics}\\
{\it University of Science and Technology Beijing, Beijing 100083, China}\\
{\it Email: zhubeibei@lsec.cc.ac.cn}\\[2mm]
Lun Ji \\
{\it LSEC, ICMSEC, Academy of Mathematics and Systems Science,}\\
{\it Chinese Academy of Sciences, Beijing 100190, China}\\
{\it School of Mathematical Sciences, University of Chinese Academy of Sciences}\\
{\it  Beijing 100049, China}\\
{\it Email: ujeybn@lsec.cc.ac.cn}
 \\[2mm]
Aiqing Zhu \\
{\it LSEC, ICMSEC, Academy of Mathematics and Systems Science,}\\
{\it Chinese Academy of Sciences, Beijing 100190, China}\\
{\it School of Mathematical Sciences, University of Chinese Academy of Sciences}\\
{\it Beijing 100049, China}\\
{\it Email: zaq@lsec.cc.ac.cn}
\\[2mm]
Yifa Tang\\
{\it LSEC, ICMSEC, Academy of Mathematics and Systems Science}\\
{\it Chinese Academy of Sciences, Beijing 100190, China}\\
{\it School of Mathematical Sciences, University of Chinese Academy of Sciences}\\
{\it Beijing 100049, China} \\
{\it Email: tyf@lsec.cc.ac.cn}}

\date{}

\maketitle

\markboth{\small B.B. Zhu, L. Ji, A.Q. Zhu and Y.F. Tang}
{\small Poisson Integrators for Poisson systems}

\begin{abstract}
We propose Poisson integrators for the numerical integration of separable Poisson systems. We analyze three situations in which the Poisson systems are separated in three ways and the Poisson integrators can be constructed by using the splitting method. Numerical results show that the Poisson integrators outperform the higher order non-Poisson integrators in phase orbit tracking, long-term energy conservation and efficiency.

\bigskip


\bigskip

\end{abstract}

\section{Introduction}

In this paper we propose the Poisson integrators for the numerical integration of the Poisson systems\cite{Lie} with
separable Hamiltonian. Poisson systems have Poisson structures which are preserved by the
Poisson integrators. There is no universal approach to constructing the Poisson integrators for arbitrary Poisson system. However, by using the splitting method, one can construct the Poisson integrators for separable Poisson systems. We identify three situations in which the Poisson systems are separated in three ways and the Poisson integrators can be constructed.

Poisson systems are generalized canonical Hamiltonian systems where the constant matrix $J^{-1}$ is replaced by a variable-dependent matrix $R(Z)$.  They have been discovered in a variety of scientific disciplines, such as the celestial mechanics, quantum mechanics, plasma physics and fluid dynamics. The well-known Poisson systems are the Euler equations for the rigid body\cite{Touma}, the nonlinear Sch\"{o}dinger equations\cite{Tang1,Faou}, the charged particle system\cite{Zhou,LiT,Zhang2}, the gyrocenter system\cite{Qin,Zhang,Zhu2}, the Maxwell-Vlasov equations\cite{LiYZ,Morrison2}, the ideal MHD equations\cite{Morrison1} and the isentropic compressible fluids. The phase flow of the Poisson system is usually very difficult to obtain. Thus, it is critical to construct accurate and efficient numerical integrators with long-term conservation property and stability. The Poisson integrators, like the symplectic methods\cite{Arnold,Channell,Feng1984,FPS,Forest,GNT,SSC94,Suris,Tang1} for canonical Hamiltonian systems, exhibit advantageous structure-preserving properties\cite{Channell,Feng1984,Forest}. Meanwhile, the Poisson integrators have the property of long-term energy conservation. Therefore, we will formulate the construction of the Poisson integrators for Poisson systems.

Many researchers have paid attention to investigating the Poisson integrators for the Poisson systems, including the theoretical results on the construction of the integrators \cite{Channell91,GeZ} and the application of the integrators to the Schr\"{o}dinger equation\cite{Faou}, the rigid body problem\cite{Touma} and the charged particle system\cite{Zhou}.
Ge and Marsden proposed the Lie-Poisson integrator that exactly preserves the Lie-Poisson structure based on the generating function which is derived as an approximate solution of Hamiltonian-Jacobi equation\cite{GeZ}. Channel and Scovel reformulate the integrator of Ge and Marsden in terms of algebra variable and implement it to arbitrary high order for regular quadratic Lie algebra\cite{Channell91}. For the application of the Poisson integrators, Faou and Lubich derived a symmetric Poisson integrator using the variational splitting technique based on the discovery that the Hamiltonian reduction of the Schr\"{o}dinger equation to the Gaussian wavepacket manifold inherits a Poisson structure\cite{Faou}.
Touma and Wisdom derived a symplectic integrator for a free rigid body and incorporated this integrator in the $n$-body integrator\cite{Wisdom} to provide a Lie-Poisson integrator for the one or more rigid bodies dynamics\cite{Touma}. Recently, the splitting technique has been applied to construct the Poisson integrators for the Poisson systems. Non-canonical Hamiltonian systems are special Poisson systems with invertible $R(Z)$. Zhu et al. investigated the particular situations that the explicit K-symplectic schemes can be constructed for the non-canonical Hamiltonian systems\cite{Zhu}. He et al. constructed the explicit K-sympletic methods for the charged particle system\cite{Zhou}. Li et al. used the Fourier spectral method and the finite volume method in space, coupled with the splitting method in time to develop the numerical methods which have good conservation property for the Vlasov-Maxwell equations\cite{LiYZ}.

In the present article we separate the Poisson systems in three ways and identify three situations in which the Poisson integrators can be constructed. By separating the Poisson system into several subsystems and exactly solving the subsystems, one can obtain a first order Poisson integrator by composing the exact solution of the subsystems. Furthermore, higher order Poisson integrator can be constructed by composing the first order Poisson integrator.  The Poisson integrators are compared with the higher order Runge-Kutta methods\cite{Sanzserna,Butcher1,Butcher2} to demonstrate their superiorities in structure preservation. The numerical simulations in two Poisson systems show that the Poisson integrators behave better in phase orbit tracking, long-term energy conservation than the higher order Runge-Kutta methods.

This paper is organized as follows. Section 2 gives a brief introduction to the Poisson systems and the Poisson integrators. Section 3 indicates how to use the splitting method to construct the Poisson integrators. We identify three situations that the Poisson integrators can be constructed. Section 4 presents two classical Poisson systems. In Section 5, numerical methods that are used to make comparison are presented and the numerical results in two Poisson systems are provided. In Section 6, we summarize our work.

\section{Poisson systems and Poisson integrators}

\newtheorem{defin}{Definition}

Poisson systems\cite{Lie} are generalizations of canonical Hamiltonian systems.
It is of the following form
\begin{equation}\label{poisson}
\frac{dZ}{dt}=R(Z)\nabla H(Z),\quad  Z=(z_1,z_2,\cdots,z_m)\in \mathbb{R}^m
\end{equation}
where $H$ is the Hamiltonian and the matrix $R(Z)=(r_{ij}(Z))$ is skew-symmetric and for all $i,j,k$\cite{GNT}
$$
\sum_{l=1}^m \Big(\frac{\partial r_{ij}(Z)}{\partial z_l}r_{lk}(Z)+\frac{\partial r_{jk}(Z)}{\partial z_l}r_{li}(Z)+\frac{\partial r_{ki}(Z)}{\partial z_l}r_{lj}(Z)\Big)=0.
$$
The Poisson bracket\cite{Lie} of two smooth functions $F,G$ is defined as
$$
\{F,G\}(Z)=\sum_{i,j=1}^m \frac{\partial F(Z)}{\partial z_i}r_{ij}(Z)\frac{\partial G(Z)}{\partial z_j}
$$
or more compactly as
$$
\{F,G\}(Z)=\nabla F(Z)^{\top}R(Z)\nabla G(Z).
$$
The Poisson bracket has the property of bilinearity\cite{GNT}
$$
\{aF+bG, H\}=a\{F,H\}+b\{G,H\},
$$
$$
\{F,aG+bH\}=a\{F,G\}+b\{F,H\},
$$
and skew-symmetry
$$
\{F,G\}=-\{G,F\}.
$$
It also satisfies the Lebniz's rule
$$
\{FG,H\}=\{F,H\}G+F\{G,H\}
$$
and the Jacobi identity
$$
\{\{F,G\},H\}+\{\{G,H\},F\}+\{\{H,F\},G\}=0.
$$
If we replace the matrix $R(y)$ with the constant matrix $J^{-1}$ where
\begin{equation*}
J=\bordermatrix{&\cr
&O_n & I_n\cr
&-I_n & O_n\cr
},
\end{equation*}
then the Poisson system becomes a canonical Hamiltonian system.

\begin{defin}
Given a transformation $\phi: U\rightarrow \mathbb{R}^m$ (where $U$ is a open set in $\mathbb{R}^m$), if its Jacobian satisfies
\begin{equation}
\Big[\frac{\partial \phi(Z)}{\partial Z}\Big]^{\top}R(\phi(Z))\Big[\frac{\partial \phi(Z)}{\partial Z}\Big]=R(Z),
\end{equation}
we call it a Poisson map\cite{GNT} with respect to the Poisson bracket defined above.
\end{defin}

As is well known that the Hamiltonian system has the symplectic structure which is exactly preserved by the symplectic geometric methods\cite{FPS,GNT,SSC94}. The Poisson system (\ref{poisson}) also has the Poisson structure which is defined by
$$
W=\sum_{1\le i,j \le m} r_{ij}(Z)dz_i\wedge dz_j.
$$
The exact phase flow $\varphi_t(Z)$ of the Poisson system is a Poisson map. As the Poisson system is usually highly nonlinear system, it is difficult to obtain its phase flow. However, one can use numerical methods that exactly preserve the Poisson structure of the Poisson system. This kind of numerical methods is called Poisson integrators.

 \begin{defin}
Given a numerical method $G_h: Z\rightarrow \tilde{Z}$, if its Jacobian satisfies
\begin{equation}
\Big[\frac{\partial G_h(Z)}{\partial Z}\Big]^{\top}R(G_h(Z))\Big[\frac{\partial G_h(Z)}{\partial Z}\Big]=R(Z),
\end{equation}
we call it a Poisson integrator\cite{GNT}.
\end{defin}
Generally, it is a difficult task to construct the Poisson integrator for general Poisson system. There is no universal approach to constructing the Poisson
integrator for arbitrary Poisson system. However, in many cases of interest, we can construct the Poisson integrators for separable Poisson systems by using the splitting method.

\section{Poisson integrators based on splitting method}

\subsection{Poisson systems that are separated into two subsystems}\label{sec3.1}

Now we introduce how to use the splitting method\cite{GNT,Zhu,Blanes} to construct the Poisson integrator. We consider the case that the Hamiltonian $H$ of the Poisson system (\ref{poisson}) is separable.

Firstly, we are concerned with the case that the Poisson system is of $2n$ dimension and
 the Hamiltonian $H(Z)$ can be separated into two parts, i.e.
$
H(Z)=H_1(z_1,\cdots,z_n)+H_2(z_{n+1},\cdots,z_{2n}).
$
Then the Poisson system can also be separated into two subsystems
\begin{equation}\label{subsystem1}
\frac{dZ}{dt}=R(Z)\nabla_{Z} H_1,
\end{equation}
\begin{equation}\label{subsystem2}
\frac{dZ}{dt}=R(Z)\nabla_{Z} H_2.
\end{equation}
If the two subsystems (\ref{subsystem1}) and (\ref{subsystem2}) can be solved exactly, then the integrators obtained by composing the exact solution of the subsystems are
the Poisson integrators of the Poisson system (\ref{poisson}). If we denote the exact solution of (\ref{subsystem1}) by $\varphi_t^1$, the exact solution of (\ref{subsystem2}) by $\varphi_t^2$, then $\varphi_t^2\circ \varphi_t^1$ is a first order Poisson integrator. Furthermore, if we use the Strang's splitting formula\cite{Strang}, then $\varphi_{t/2}^1 \circ \varphi_{t}^2\circ \varphi_{t/2}^1$ is a second order Poisson integrator. There are many other composing techniques that help to improve the order of the method. One commonly used method is the symmetric composition of first order methods. Given a first order Poisson integrator $\Phi_h$ where $h$ represents the time stepsize, we can compose it by a symmetric way\cite{GNT}
$$
\Psi_h\equiv \Phi_{\alpha_sh}\circ \Phi_{\beta_s h}^*\circ \cdots \circ \Phi_{\beta_2 h}^* \Phi_{\alpha_1 h}\circ \Phi_{\beta_1 h}^*
$$
to make the method $\Psi_h$ a higher order symmetric method. The coefficients satisfy $\alpha_{i}=\beta_{s-i}, 1\le i\le s$. The method $\Phi_h^*$ represents the adjoint method of $\Phi_h$.

 The problem is that under which circumstance, the two subsystems can be solved exactly and the Poisson integrators can be constructed. We identify the situation in which the two subsystems are solvable and the results are listed in the following theorem. To simplify the notations, we denote $(z_1,\cdots,z_n)=(p_1,\cdots,p_n)$ and $(z_{n+1},\cdots,z_{2n})=(q_1,\cdots,q_n)$.

 Here solvable means that each subsystem can be explicitly solved or solved as $2n$ algebraic equations.

\newtheorem{thm}{Theorem}
\begin{thm}
\label{theorem1}
The two subsystems are solvable in the following situation:

The matrix $R$ has the form of
 $$R=
\bordermatrix{&\cr
&O_n & A\cr
&-A^{\top} & O_n\cr
},$$
where $A=(a_{ij})_{n\times n}$, and $a_{ij}$'s are continuous functions of $p_i$ and $q_j$ for any $1\leqslant i,j\leqslant n$.
\end{thm}
\begin{proof}
We only consider solving (\ref{subsystem1}), and (\ref{subsystem2}) can be solved in the similar way.

Under these conditions, for any $1\leqslant i\leqslant n$, (\ref{subsystem1}) shows
$$\frac{dp_i}{dt}=0,~~\frac{dq_i}{dt}=-\sum\limits_{j=1}^na_{ji}(p_j,q_i)\frac{\partial H_1}{\partial p_j}.$$
 Thus we have $p_i\equiv p_i(0)$, which shows $\frac{\partial H_1}{\partial p_j}$, $1\le j\le n$ are all constants since $H_1$ is a function of all $p_j$'s. Therefore, $-\sum\limits_{j=1}^na_{ji}(p_j,q_i)\frac{\partial H_1}{\partial p_j}$ is just a function of $q_i$. Let $f_i(q_i)=-\sum\limits_{j=1}^na_{ji}(p_j,q_i)\frac{\partial H_1}{\partial p_j}$, i.e. $\frac{dq_i}{dt}=f_i(q_i)$, thus $\int_{q_i(0)}^{q_i(t)}\frac{1}{f_i(q)}dq=t$. Then we discuss how to solve the integral equation for each $q_i$ in the following three cases.

(i) If $f_i(q_i(0))=0$, we know $q_i\equiv q_i(0)$ is a solution of (\ref{subsystem1}). The solution is unique when $f_i$ is Lipschitz continuous.

(ii) If $f_i(q_i(0))>0$, then $q_i(t)>q_i(0)$ for some small $t$ since $t>0$. Let $s$ to be the smallest number of $q$ satisfying $f_i(q)=0$(set $s=+\infty$ if $f_i>0$ on $(q_i(0),+\infty)$), and take $F_i(x)=\int_{q_i(0)}^x\frac{1}{f_i(q)}dq,x\in(q_i(0),s)$, then $F_i(q_i(t))=t$. Furthermore, since $f_i(q_i(0))>0$ on $(q_i(0),s)$, we know $F_i$ is strictly increasing on $(q_i(0),s)$, so it has an inverse function $F_i^{-1}(x)$. Now we make a further classification:

$1^{\comp}$ If $F_i(s)=+\infty$, no matter $s$ is finite or not, we  always have $q_i(t)=F_i^{-1}(t)$;

$2^{\comp}$ If both $s$ and $F_i(s)$ are finite, then we have $q_i(t)=F_i^{-1}(t)$ when $t\in(0,F_i(s))$. As $f_i(s)=0$($f_i$ is continuous), it is similar to the case (i) when $f_i(q_i(0))=0$. So we have one unique solution on $(F_i(s),+\infty)$, i.e. $q_i(t)\equiv s$, when $f_i$ is Lipschitz continuous on $[s,+\infty)$.

$3^{\comp}$ If $s=+\infty$ but $F_i(s)$ is finite, then we have $q_i(t)=F_i^{-1}(t)$ when $t\in(0,F_i(s))$. When $t\geqslant F_i(s)$, $q_i(t)$ does not exist.

(iii) If $f_i(q_i(0))<0$, the discussion is similar to the case (ii).
\end{proof}

\subsection{Poisson systems that are separated into $n+1$ subsystems}
\label{sec3.2}

We consider the Poisson system of $2n$ dimension and denote by $(z_1,\cdots,z_n)=(p_1,\cdots,p_n)$ and $(z_{n+1},\cdots,z_{2n})=(q_1,\cdots,q_n)$. We assume that the Hamiltonian $H$ can be separated into $H=H_1(p_1,\cdots,p_n)+H_2(q_1)+H_3(q_2)+\cdots+H_{n+1}(q_n)$, then the Poisson system can be separated into $n+1$ subsystems
\begin{equation}\label{2ndsubsystem1}
\frac{dZ}{dt}=R(Z)\nabla_{Z} H_1,
\end{equation}
\begin{equation}
\label{2ndsubsystem2}
\frac{dZ}{dt}=R(Z)\nabla_{Z} H_2,
\end{equation}
\begin{equation*}
\cdots\cdots\cdots
\end{equation*}
\begin{equation}
 \label{2ndsubsystem3}
\frac{dZ}{dt}=R(Z)\nabla_{Z} H_{n+1}.
\end{equation}

We identify the situation where the above $n+1$ subsystems are solvable and the Poisson Integrators can be constructed. Here the meaning of solvable is the same as before in Section \ref{sec3.1}.

\begin{thm}\label{theorem2}
All the above $n+1$ subsystems are solvable in the following situation:

The matrix $R$ has the form of
 $$R=
\bordermatrix{&\cr
&O_n & A\cr
&-A^{\top} & C\cr
},$$
where the matrix $A=(a_{ij})_{n\times n}$ and $C=(c_{ij})_{n\times n}$. The elements $a_{ij}$'s are just continuous functions of $q_j$ and the elements $c_{ij}$'s are continuous functions of $p_1,p_2,\cdots,p_n$ and $q_j$ for any $1\leqslant i,j\leqslant n$.
\end{thm}

\begin{proof}
Here we only proof how to solve the subsystem (\ref{2ndsubsystem1}) and (\ref{2ndsubsystem2}).

 For the subsystem (\ref{2ndsubsystem1})
\begin{equation}\label{subsystem2nd3}
\left\{
\begin{split}
\frac{dp_i}{dt}&=0,\quad 1\le i\le n\\
\frac{dq_{j}}{dt}&=-\sum_{k=1}^n a_{kj}(q_j)\frac{\partial H_1}{\partial p_k},\quad 1\le j\le n.
\end{split}
\right.
\end{equation}
As $p_1, p_2,\cdots,p_n$ are all constants and $H_1$ is a continuous function of all $p_k$'s, then $\frac{\partial H_1}{\partial p_k}, 1\le k\le n$ are also constants. Therefore, $-\sum_{k=1}^n a_{kj}(q_j)\frac{\partial H_1}{\partial p_k}, 1\le j \le n$ are just continuous functions of $q_j$. According to the proof in Theorem \ref{theorem1}, each $q_j$ for $1\le j \le n$ is solvable.

The subsystem (\ref{2ndsubsystem2}) is
\begin{equation}\label{subsystem2nd4}
\left\{
\begin{split}
\frac{dp_i}{dt}&=a_{i1}(q_1)\frac{\partial H_2}{\partial q_1},\quad 1\le i\le n\\
\frac{dq_{1}}{dt}&=0,\\
\frac{dq_j}{dt}&=c_{j1}(p_1,p_2,\cdots,p_n,q_1)\frac{\partial H_2}{\partial q_1},\quad 2\le j\le n.
\end{split}
\right.
\end{equation}
From the $n+1$-th equality of Equation (\ref{subsystem2nd4}), we derive that $q_1(t)\equiv q_{10}$ where $q_{10}$ is the initial value. Then we can easily know that $a_{i1}(q_1)\frac{\partial H_2}{\partial q_1}, 1\le i\le n$ are all constants. As a result, each $p_i$ for $1\le i \le n$ can be solved explicitly. As the time derivative of $q_j, 2\le j\le n$ does not depend on $q_j$, and all $p_i$'s can be solved explicitly, therefore $q_i$ can be solved exactly with a given initial value $q_{j0}$, i.e.
$$
q_j(t)=q_{j0}+\frac{\partial H_2}{\partial q_{10}}\int_0^t c_{j1}(p_1(\xi),p_2(\xi),\cdots, p_n(\xi),q_{10})d\xi, \quad 2\le j\le n.
$$
The way to solve the other subsystem is similar to that of the subsystem (\ref{2ndsubsystem2}). The proof is completed.
\end{proof}

The case that the Hamiltonian $H$ can be separated into $H=H_1(p_1)+H_2(p_2)+\cdots+H_{n}(p_n)+H_{n+1}(q_1,q_2,\cdots,q_n)$ is similar, the Poisson system can be separated into $n+1$ subsystems as well. For such a case, we easily know that if the matrix $R$ is of the form
$$R=
\begin{pmatrix}
C & A\\
-A^{\top}& O
\end{pmatrix}$$
with the matrix $A=(a_{ij})_{n\times n}$ and $C=(c_{ij})_{n\times n}$ being the situation in Theorem \ref{theorem2}, then the $n+1$ subsystems are all solvable.

\subsection{Poisson systems that are separated into $m$ subsystems}

In this subsection, we consider arbitrary dimensional Poisson system. If the Hamiltonian $H$ of the $m$ dimensional Poisson system is totally separable with respect to each argument $z_i$, i.e. $H(z_1,z_2,\cdots,z_m)=H_1(z_1)+H_2(z_2)+\cdots+H_m(z_m)$, then the Poisson system can be separated into $m$ subsystems
\begin{equation}\label{subsystem3}
\frac{dZ}{dt}=R(Z)\nabla_{Z} H_1,
\end{equation}
\begin{equation}\label{subsystem4}
\frac{dZ}{dt}=R(Z)\nabla_{Z} H_2,
\end{equation}
$$
\cdots\cdots\cdots
$$
\begin{equation}\label{subsystem5}
\frac{dZ}{dt}=R(Z)\nabla_{Z} H_m.
\end{equation}
We identify the situation in which the above $m$ subsystems are solvable and the Poisson Integrators can be constructed. Here the meaning of solvable is the same as before in Section \ref{sec3.1}.

\begin{thm}\label{theorem3}
In the following situation the above $m$ subsystems are solvable:

The skew-symmetric matrix $R$ has the form of
 $$R=
\begin{pmatrix}
0	  &	  r_{12}(z_1,z_2)&	  r_{13}(z_1,z_2) & \cdots &r_{1m}(z_1,z_m)  \\	
-r_{12}(z_1,z_2)&0  &	  r_{23}(z_2,z_3) & \cdots &	  r_{2m}(z_2,z_m)\\
-r_{13}(z_1,z_3)&-r_{23}(z_2,z_3) &0 & \cdots &	  r_{3m}(z_3,z_m)\\
\vdots &\vdots & \vdots & \cdots & \vdots\\
-r_{1m}(z_1,z_m) & -r_{2m}(z_2,z_m) & -r_{3m}(z_3,z_m) & \cdots & 0
\end{pmatrix}$$
where $R=(r_{ij})_{n\times n}$, $r_{ij}$'s are continuous functions of $z_i$ and $z_j$ for any $1\leqslant i,j\leqslant n$.
\end{thm}

\begin{proof}
We only consider solving the subsystem (\ref{subsystem3}), the other $m-1$ subsystems can be solved similarly.
As the Hamiltonian $H_1$ of the subsystem (\ref{subsystem3}) just depends on $z_1$, then the subsystem can be written as
\begin{equation}\label{subsystem11}
\left\{
\begin{split}
\frac{dz_1}{dt}&=0,\\
\frac{dz_i}{dt}&=-r_{1i}(z_1,z_i)\frac{\partial H_1}{\partial z_1},\quad 2\le i\le m.
\end{split}
\right.
\end{equation}
The first equation of (\ref{subsystem11}) implies that $z_1=Const$, then $\frac{\partial H_1}{\partial z_1}$ is also a constant. As $z_1$ is a constant, then $r_{1i}(z_1,z_i), 2\le i\le m$ in (\ref{subsystem11}) are just continuous functions of $z_i$. Therefore, according to the proof in Theorem \ref{theorem1}, $z_2,z_3,\cdots,z_m$ in (\ref{subsystem11}) are all solvable.
\end{proof}

We will construct the Poisson integrators for two Poisson systems to verify our theoretical results. The numerical results will be shown in Section \ref{numerexper}.

\section{Two Poisson systems}

\subsection{Charged particle system}

Dynamics of charged particles\cite{Zhou,LiT,Zhang2} in external electromagnetic fields plays a fundamental role in plasma physics. The fast gyromotion and the slow gyrocenter motion
constitute the two components of the dynamics of one charged particle in  magnetized plasma. If one averages out the fast gyromotion from the charged particle motion,
the behaviour of gyrocenters is governed by gyrokinetics and related theories. The motion of the charged particle in a given electromagnetic field $(E(X), B(X))$ is governed by the Lorentz force law. If we denote the position variable of the charged particle by $X$ and its velocity by $V$, then the charged particle motion can be expressed as a 6 dimensional Poisson system under the variable $Z=(X,V)^{\top}=(x_1,x_2,x_3,v_1,v_2,v_3)^{\top}$
$$
Z=R(Z)\nabla H(Z)
$$
where
$$
R(Z)=\bordermatrix{&\cr
&O & \frac{I}{m}\cr
&-\frac{I}{m} & -\frac{q\hat{B}(X)}{m^2}\cr
},
$$
and the Hamiltonian is $H(X,V)=mv_1^2/2+mv_2^2/2+mv_3^2/2+q\varphi(X)$ with the scalar potential $\varphi(X)$. The electronic field is $E(X)=-\nabla \varphi$, the magnetic field is $B(X)=(B_1(X),B_2(X),B_3(X))$ and
 the matrix $\hat{B}(X)$ is
 $$
 \hat{B}(X)=\bordermatrix{&\cr
 & 0 &-B_3(X) & B_2(X)\cr
 & B_3(X) & 0 & -B_1(X)\cr
 & -B_2(X) & B_1(X) &0 \cr
 }.
 $$

\subsection{Gyrocenter system}

We then introduce the gyrocenter system \cite{Qin,Zhang,Zhu2} with the variable $Z=(X,u)^{\top}$, where $X=(x,y,z)^\top$ is the 3-dimensional position variable of the gyrocenter. Note that $A(X)$ is the vector potential of the magnetic field, and $B(X)$ is the magnetic field.
The relationship between $A(X)$ and $B(X)$ is $B(X)=\nabla\times A(X)$.

We assume that $A(X)=(f,g,h)^\top$ where $f,g,h$ are all smooth functions of the three arguments $x,y,z$. The notation $f_x$ represents the derivative of $f$ with respect to $x$. Then $B(X)=\nabla\times A(X)=(h_y-g_z,f_z-h_x,g_x-f_y)^\top$. The unit vector along the direction of the magnetic field is $b(X)=(b_1,b_2,b_3)^\top=\dfrac{B(X)}{|B(X)|}$.

The Lagrangian of the gyrocenter system
\begin{equation}\nonumber
L(X,\dot{X},u,\dot{u})=[A(X)+ub(X)]\cdot\dot{X}-[\frac{1}{2}u^{2}+\mu B(X)+\varphi(X)],
\end{equation}
is first given by Littlejohn\cite{Littlejohn}. The Euler-Lagrange equations of the Lagrangian with respect to $X$ and $u$ result in the gyrocenter motion which can be expressed as
\begin{equation}\label{KVH}
K({Z})\dot{{ Z}}=\nabla H({Z}),
\end{equation}
where $H(Z)=\frac{1}{2}u^{2}+\mu |B(X)|+\varphi({X})$ is the Hamiltonian with the scalar potential $\varphi({X})$, and the skew-symmetric matrix $K(Z)$ is
\begin{equation}\nonumber
K(Z)=\bordermatrix{&\cr
 &0&a_{12}&a_{13}&-b_{1}\cr
 &-a_{12}&0&a_{23}&-b_{2}\cr
 &-a_{13}&-a_{23}&0&-b_{3}\cr
 &b_{1}&b_{2}&b_{3}&0\cr
 }
\end{equation}
with the elements being
$$a_{12}=g_x-f_y+u(\dfrac{\partial b_2}{\partial x}-\dfrac{\partial b_1}{\partial y}),$$ $$a_{13}=h_x-f_z+u(\dfrac{\partial b_3}{\partial x}-\dfrac{\partial b_1}{\partial z}),$$ $$a_{23}=h_y-g_z+u(\dfrac{\partial b_3}{\partial y}-\dfrac{\partial b_2}{\partial z}),$$
If the matrix $K(Z)$ is invertible, i.e. $\det(K(Z))=\left|a_{13}b_1-a_{13}b_2+a_{12}b_3\right|^2\neq 0$, then the gyrocenter system (\ref{KVH}) becomes a Poisson system with
$$R(Z)=\dfrac{1}{a_{12}b_3-a_{13}b_2+a_{23}b_1}
\begin{pmatrix}
0 & -b_3 & b_2 & a_{23} \\ b_3 & 0 & -b_1 & -a_{13} \\
-b_2 & b_1 & 0 & a_{12} \\ -a_{23} & a_{13} & -a_{12} & 0
\end{pmatrix}.$$

\section{Numerical Experiments}
\label{numerexper}

\subsection{Numerical methods}

Denote $\Phi_h$ by the first order Poisson integrator which is composed by the exact solution of the subsystems. Five numerical methods will be applied to do numerical simulation for the above two Poisson systems.

2ndEPI: the second order Poisson integrator\cite{Strang}, which is the composition of $\Phi_h$ and its adjoint method
$$
\Psi_h^2\equiv \Phi_{h/2}^*\circ \Phi_{h/2}.
$$

4thEPI1: the fourth order Poisson integrator, which is
$$
\Psi_h^4\equiv \Phi_{\alpha_5h}\circ \Phi_{\beta_5 h}^*\circ \cdots \circ \Phi_{\beta_2 h}^*\circ \Phi_{\alpha_1 h}\circ \Phi_{\beta_1 h}^*.
$$
The values of the parameters $\alpha_1,\beta_1,\cdots,\alpha_5,\beta_5$ are given in \cite{McLachlan}.

4thEPI2: the fourth order Poisson integrator, which is
$$
\Upsilon_h^4\equiv \Phi_{\alpha_6h}\circ \Phi_{\beta_6 h}^*\circ \cdots \circ \Phi_{\beta_2 h}^*\circ \Phi_{\alpha_1 h}\circ \Phi_{\beta_1 h}^*
$$
The values of the parameters $\alpha_1,\beta_1,\cdots,\alpha_6,\beta_6$ are given in \cite{Blanes}.

4thloba: the fourth order Runge-Kutta method based on the Lobatto quadrature\cite{Butcher1}. We denote this method by $L_h^4$.

6thloba: the sixth order Runge-Kutta method based on the Lobatto quadrature\cite{Butcher2}. We denote this method by $L_h^6$.
%

To show the advantages of the Poisson integrators in structure preservation, we compare them with the numerical methods two orders higher than theirs. We will compare the second order Poisson integrator 2ndEPI with the fourth order Runge-Kutta method 4thloba. Two fourth order Poisson Integrators 4thEPI1 and 4thEPI2 will be compared with the sixth order Runge-Kutta method 6thloba. Their behaviors in preserving the phase orbit and the energy of the system will be demonstrated in the next subsection.

\subsection{Numerical experiments for charged particle system}

We here report a few numerical experiments for two instances of the motion of one charged particle.

\textbf{Example 1:}
We choose the magnetic field to be
$B(\textbf{X})=[0,0,x_1^2+x_2^2]^{\top}$. The electronic field is set to be $E(\textbf{X})=\frac{10^{-3}}{\left(\sqrt{x_1^2+x_2^2}\right)^{3}}[x_1,x_2,0]^{\top}$ and the constants $m$ and $q$ are both set to be 1. Thus the Hamiltonian is
$H=\frac{1}{2}(v_1^2+v_2^2+v_3^2)+\frac{10^{-3}}{\sqrt{x_1^2+x_2^2}}$. Because the Hamiltonian is of the same form as the case of Section \ref{sec3.2}, we can separate the original system into four subsystems with $H_1=\frac{1}{2}v_1^2$, $H_2=\frac{1}{2}v_2^2$, $H_3=\frac{1}{2}v_3^2$ and $H_4=\frac{10^{-3}}{\sqrt{x_1^2+x_2^2}}$, respectively. It can be easily verified that under the above magnetic field, the matrix $R(Z)$ in the charged particle system satisfies the requirements in Theorem \ref{theorem2}, thus all the subsystems can be solved exactly.

For the first subsystem with $H_1=\frac{1}{2}v_1^2$, the variable $v_2(t)$ can be solved as
\begin{eqnarray*}
v_2(t)&=&v_{20}-v_{10}\int_0^t B_3(x_{10}+\xi v_{10},x_{20},x_{30})d\xi\\
&=&v_{20}-x_{10}v_{10}^2 t^2-x_{10}^2v_{10}t-\frac{1}{3}v_{10}^3t^3-v_{10}x_{20}^2 t,
\end{eqnarray*}
where $v_{10}, v_{20}, x_{10}, x_{20}, x_{30}$ represent the initial values of $v_1, v_2, x_1, x_2, x_3$. As all the subsystems can be solved explicitly, we can construct the explicit Poisson integrators.

The initial condition for the numerical simulation is chosen as $x_0=[0.5,-1,0]^{\top}$, $v_0=[0.1,0.1,0]^{\top}$. The numerical results for the five numerical methods are displayed in Figure \ref{epmfigure1}-\ref{epmfigure3}. We first simulate the charged particle motion using the methods 2ndEPI, 4thEPI2, 4thloba and 6thloba. The orbits in $x_1-x_2$ plane are displayed in Figure \ref{epmfigure1}. We can see that the orbits obtained by the two Poisson integrators 2ndEPI, 4thEPI2 are more accurate than the two Runge-Kutta methods 4thloba, 6thloba. Especially, the orbit obtained by the 4thloba method is much coarser than the 2ndEPI and 4thEPI2 method. To illustrate the order of the Poisson integrators, we
plot in Figure \ref{epmfigure2} the global errors of the variables $X=(x_1,x_2,x_3)$ and $V=(v_1,v_2,v_3)$. It is clearly shown that the 2ndEPI method is of order 2 and the 4thEPI1 and 4thEPI2 methods are of order 4. The global errors of the two fourth methods 4thEPI1, 4thEPI2 are nearly the same. The evolutions of the energy using different methods are shown in Figure \ref{epmfigure3}. The energy error of the second order Poisson integrator 2ndEPI can be bounded in a small interval while the energy error of the fourth order Runge-Kutta method 4thloba increases linearly along time. It can seen from Figure \ref{epmfigure3} that the energy errors obtained by the two fourth order Poisson integrators 4thEPI1, 4thEPI2 are much smaller than that of the sixth order Runge-Kutta method 6thloba. The energy errors of the methods 4thEPI1 and 4thEPI2 can be preserved at a very small number over long time, but that of the 6thloba method increases without bound. The numerical results clearly show the advantages of the Poisson integrators in tracking the phase orbit, preserving the energy over long time compared with the higher order Runge-Kutta methods.

\begin{figure}
\begin{center}
\subfigure[ ]{
\includegraphics[width=0.48\textwidth]{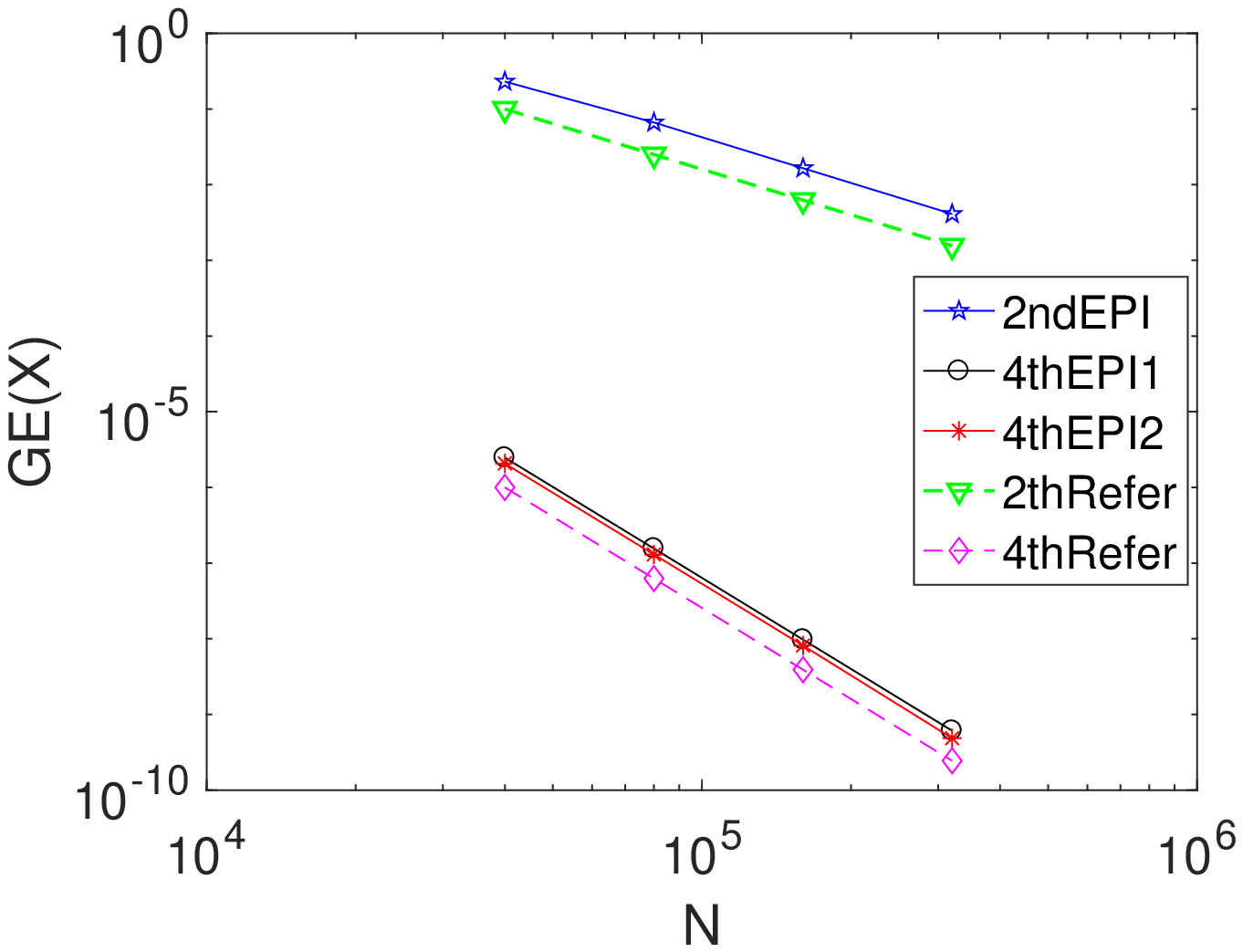}}
\subfigure[ ]{
\includegraphics[width=0.48\textwidth]{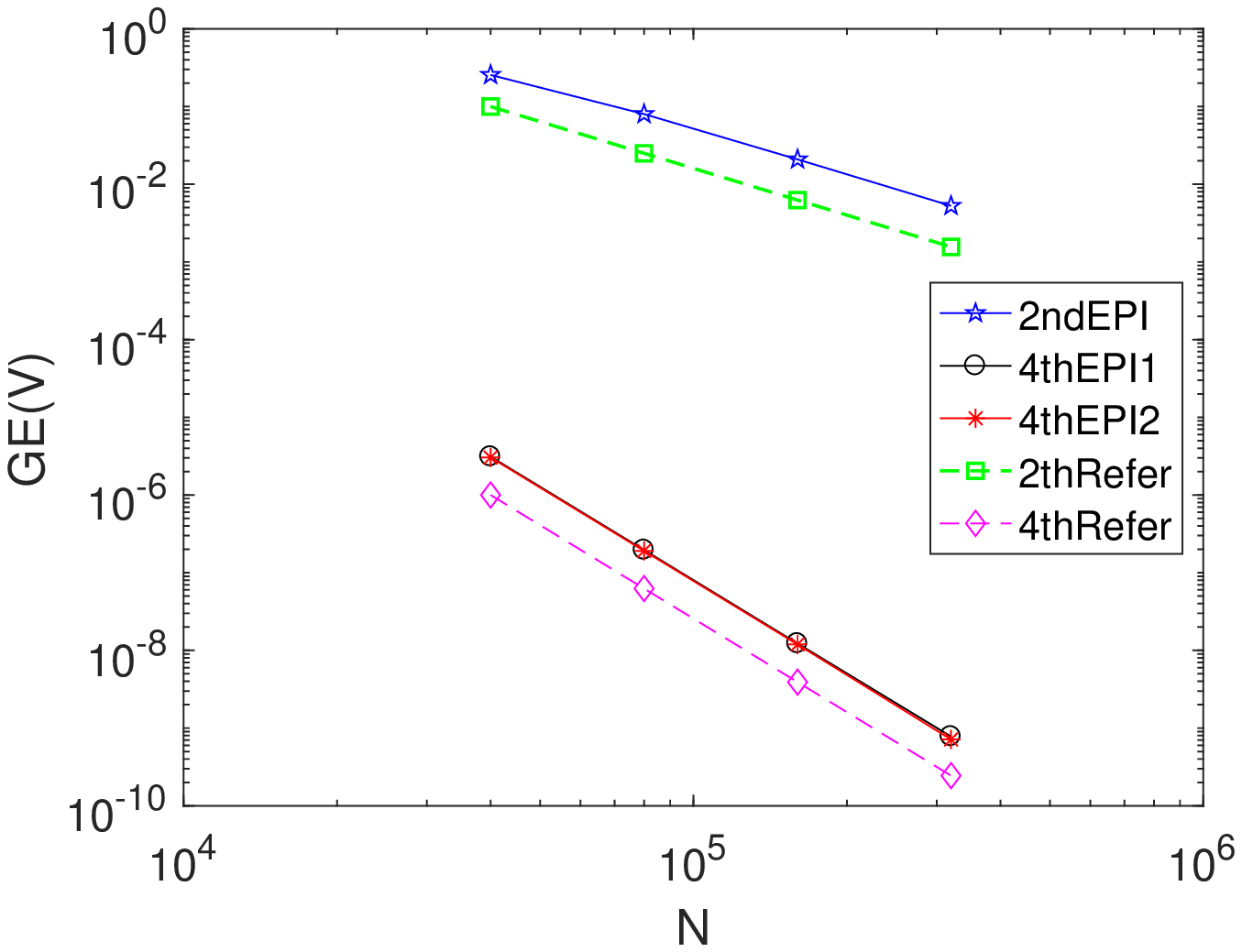}}
\end{center}
\caption{The global errors of $X$ and $V$ against
the time steps $N$ for methods 2ndEPI, 4thEPI1 and 4thEPI2 under different stepsize $h=\pi/20/2^i (i=1,2,3,4)$ in Example 1 of the charged particle system. Here the final time $T=1000\pi$. $GE(X)=max_{1\le i\le N} \parallel X_i \parallel_2$. Dashed lines are the
reference lines showing the corresponding convergence orders.
Subfigures (a) shows the global errors of the variable
$X$ while subfigures (b) shows the global errors of the variable $V$.}
\label{epmfigure1}
\end{figure}

\begin{figure}
\begin{center}
\subfigure[ ]{
\includegraphics[width=0.48\textwidth]{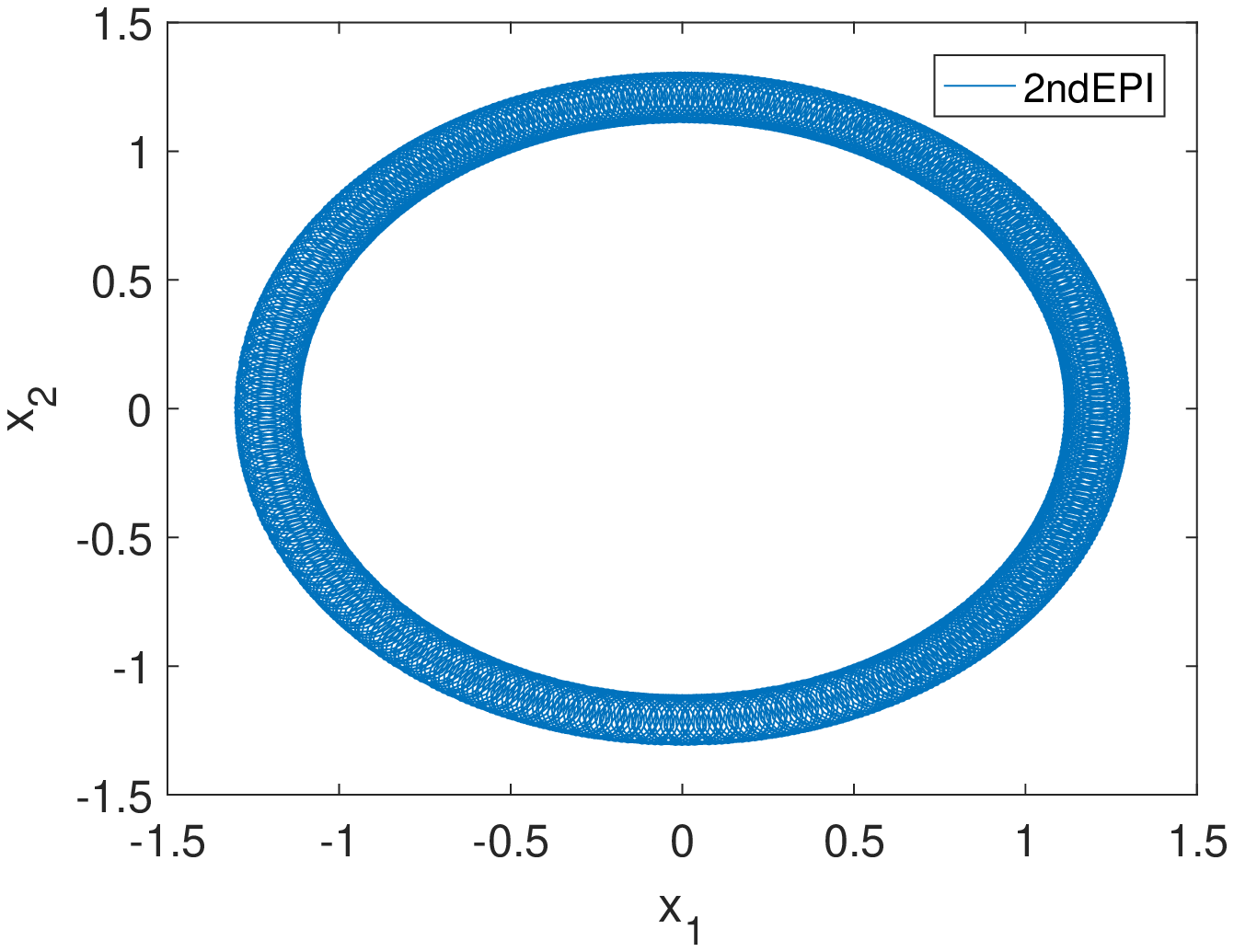}}
\subfigure[ ]{
\includegraphics[width=0.48\textwidth]{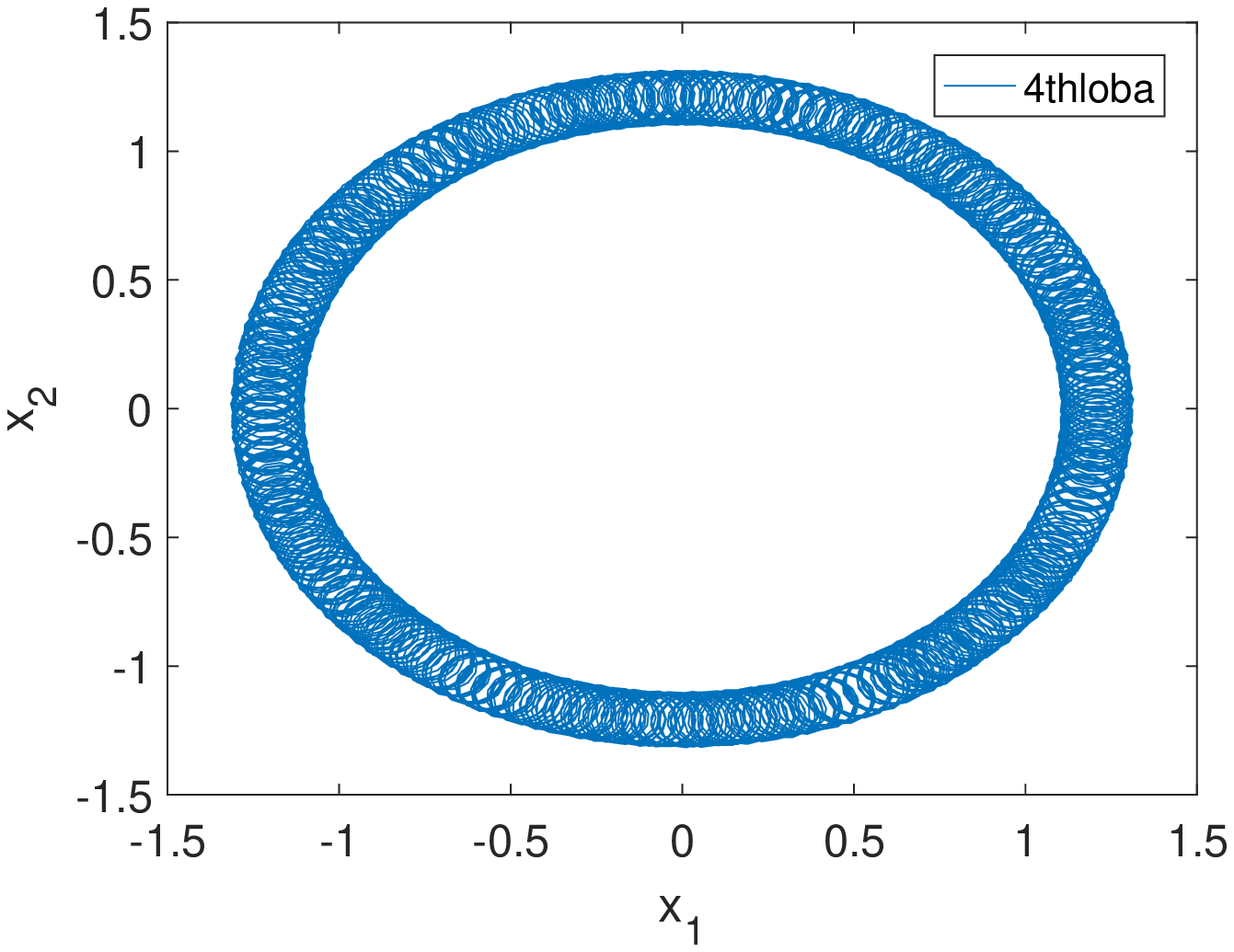}}
\subfigure[ ]{
\includegraphics[width=0.48\textwidth]{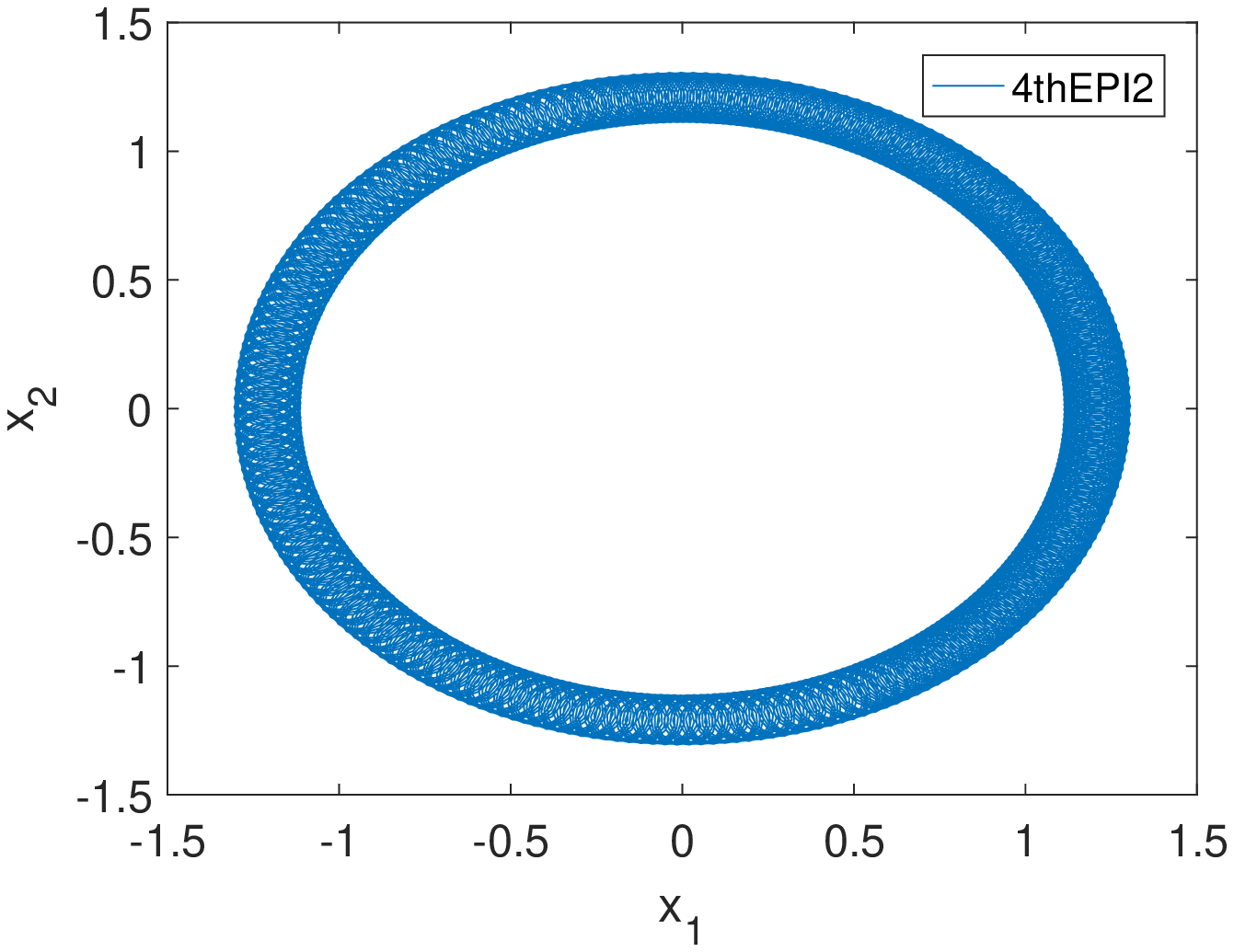}}
\subfigure[ ]{
\includegraphics[width=0.48\textwidth]{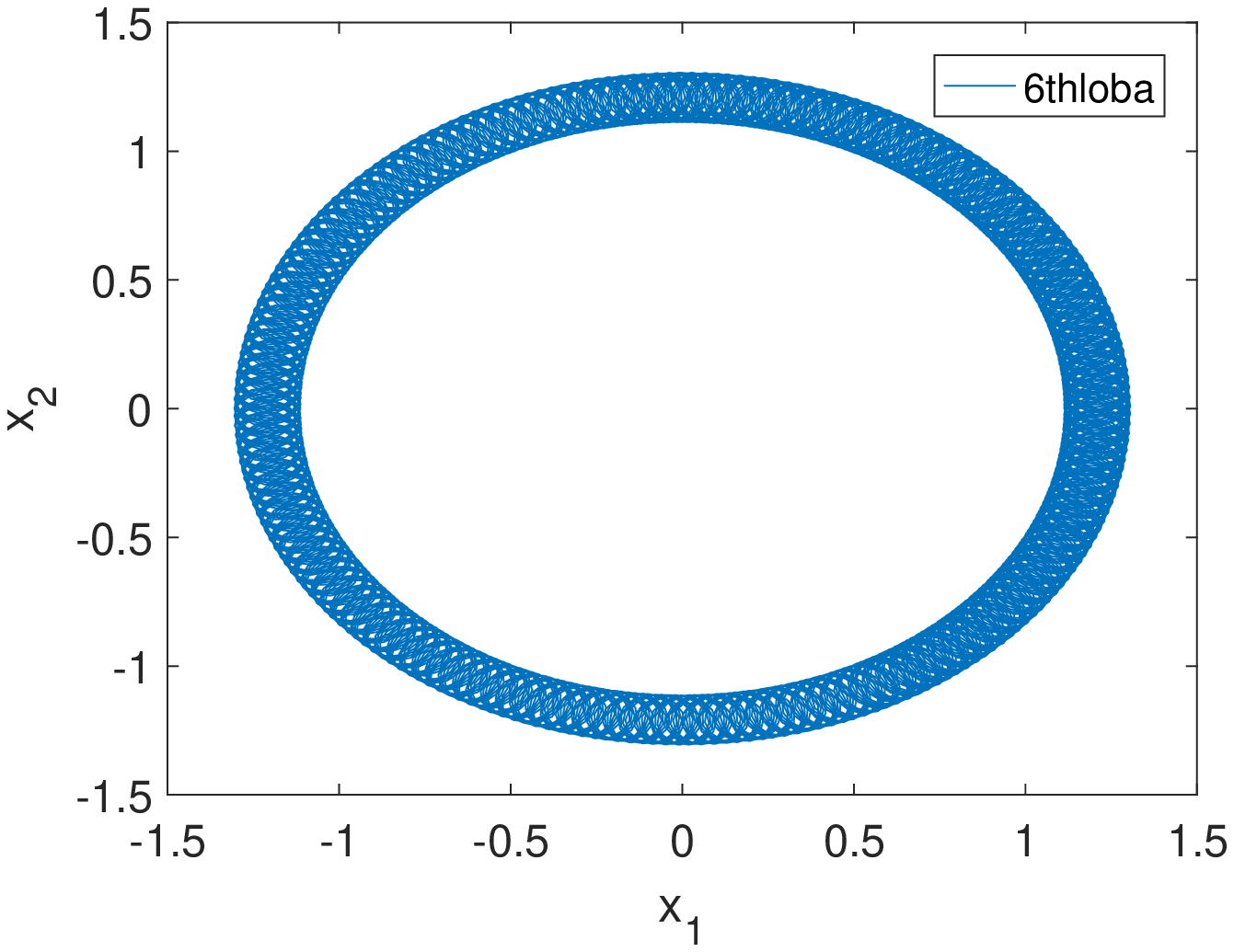}}
\end{center}
\caption{The charged particle orbit in $x_1$-$x_2$ plane simulated by
using the two Poisson integrators and the two Runge-Kutta methods over the interval
$[0,1000\pi]$. The stepsize $h$ is chosen to be $\pi/10$. Subfigure (a), (b), (c) and (d) display the orbit obtained by the 2ndEPI method,
the 4thloba method, the 4thEPI2 method and
the 6thloba method, respectively.}
\label{epmfigure2}
\end{figure}

\begin{figure}
\begin{center}
\subfigure[ ]{
\includegraphics[width=0.48\textwidth]{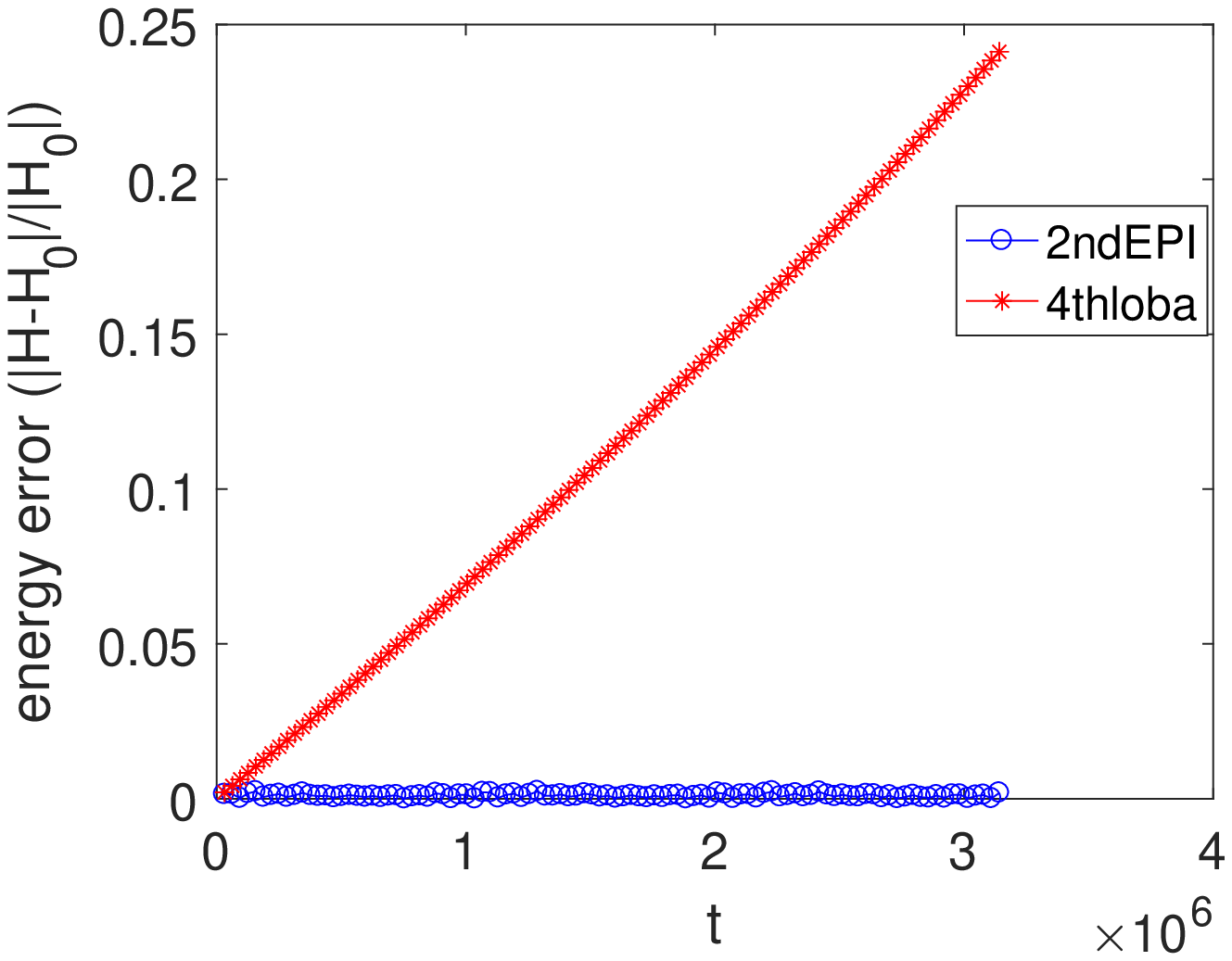}}
\subfigure[ ]{
\includegraphics[width=0.48\textwidth]{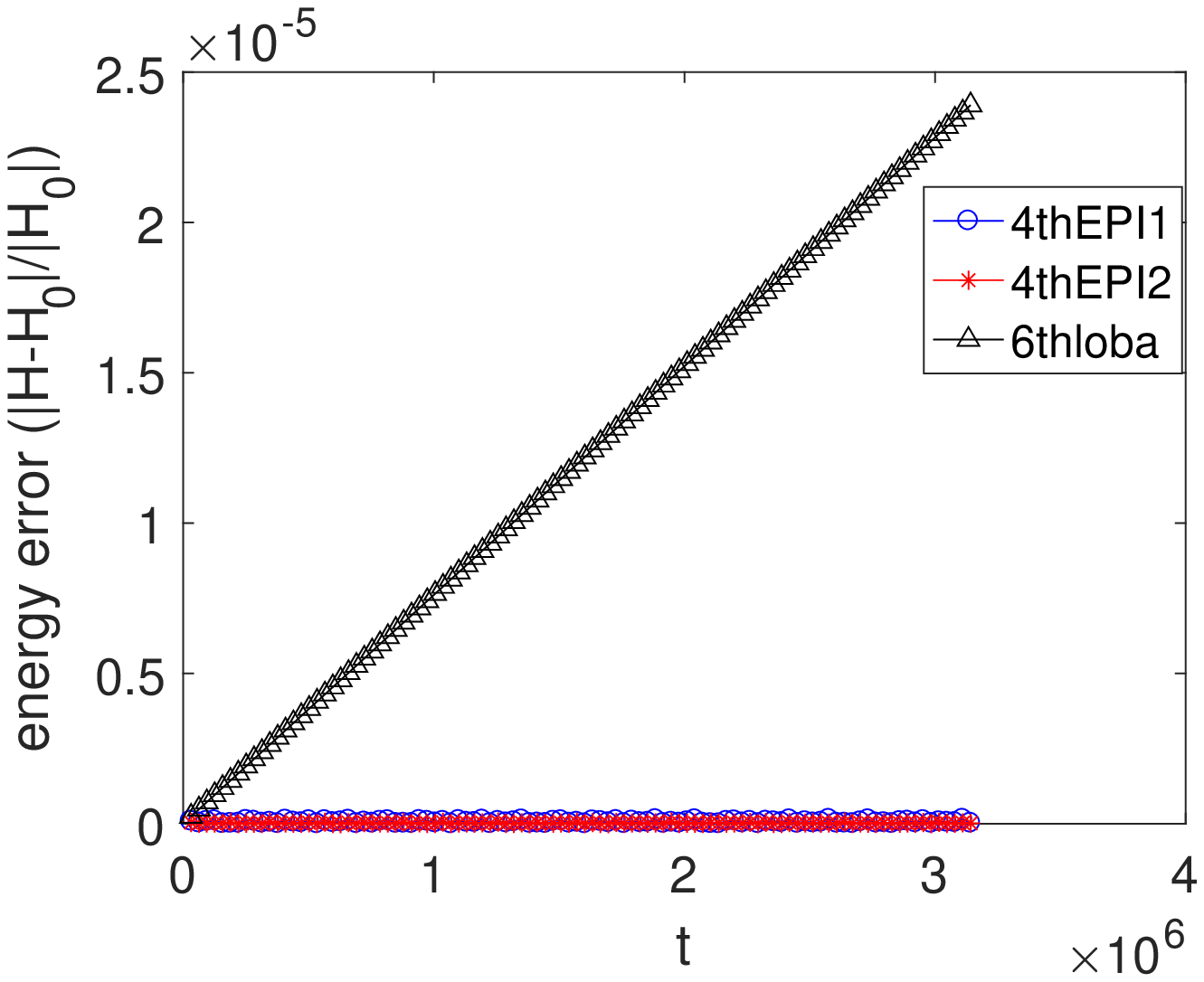}}
\end{center}
\caption{The relative energy error against $t$ for the three Poisson integrators and the two Runge-Kutta methods in Example 1 of the charged particle system. The energy
error is represented by $|H(Z_n)-H(Z_0)|/|H(Z_0)|$. The stepsize is $h=\pi/40$. Subfigure (a) displays the energy errors of the 2ndEPI method and the 4thloba method
over the time interval $[0,10^6\pi]$. Subfigure (b) displays the energy errors of the 4thEPI1 method, 4thEPI2 method and the 6thloba method over the time interval $[0,10^6\pi]$.}
\label{epmfigure3}
\end{figure}

\textbf{Example 2:}
We choose another electronic field
$E=10^{-4}\Big[\frac{1}{x_1},\frac{1}{x_2},\frac{2}{x_3}\Big]^{\top}$
and magnetic field
$B(X)=\Big[-\frac{x_3}{\sqrt{x_2^2+x_3^2}},-\frac{x_1}{\sqrt{x_1^2+x_3^2}},-\frac{x_2}{\sqrt{x_1^2+x_2^2}}\Big]$.  Thus the Hamiltonian is $H=\frac{1}{2}(v_1^2+v_2^2+v_3^2)+10^{-4}\ln(x_1)+10^{-4}\ln(x_2)+2\cdot10^{-4}\ln(x_3)$. We can easily verify that under such circumstance, the matrix $R(Z)$ satisfies the situation in Theorem \ref{theorem2}. Therefore, the original system can be separated into 4 subsystems. As the Hamiltonian function is totally separable, the original system can also be separated into 6 subsystems.

For the first subsystem with the Hamiltonian $H_1=\frac{1}{2}v_1^2$, the exact solutions for the variables $v_2(t)$ and $v_3(t)$ are
\begin{eqnarray*}
v_2(t)&=&v_{20}-v_{10}\int_0^t B_3(x_{10}+\xi v_{10},x_{20},x_{30})d\xi\\
&=&v_{20}+\frac{v_{10}x_{20}}{\sqrt{v_{10}^2}}\ln\Big(\frac{v_{10}^2t+x_{10}v_{10}}{\sqrt{v_{10}^2}}+\sqrt{x_{20}^2+(x_{10}+v_{10}t)^2}\Big)\\
& &-\frac{v_{10}x_{20}}{\sqrt{v_{10}^2}}\ln\Big(\frac{x_{10}v_{10}}{\sqrt{v_{10}^2}}+\sqrt{x_{20}^2+x_{10}^2}\Big)\\
v_3(t)&=&v_{30}+v_{10}\int_0^t B_2(x_{10}+\xi v_{10},x_{20},x_{30})d\xi\\
&=&v_{30}-\sqrt{(x_{10}+tv_{10})^2+x_{30}^2}+\sqrt{x_{10}^2+x_{30}^2}.
\end{eqnarray*}
where $v_{10}, v_{20}, v_{30}, x_{10}, x_{20}, x_{30}$ represent the initial values of $v_1, v_2, v_3, x_1, x_2, x_3$. The explicit Poisson integrators can be constructed as all the subsystems can be solved explicitly.

We perform the numerical simulation under the initial condition $x_0=[1,2,1]^{\top}$, $v_0=[1,2,2]^{\top}$. The numerical results for the five numerical methods are displayed in Figure \ref{epmfigure4}-\ref{epmfigure5}. The global errors of the variables $X=(x_1,x_2,x_3)$ and $V=(v_1,v_2,v_3)$ for the three explicit Poisson integrators 2ndEPI, 4thEPI1 and 4thEPI2 are plotted in Figure \ref{epmfigure4}. The orders of the three methods are clearly shown in Figure \ref{epmfigure4}.
The energy evolutions of different Poisson integrators and different Runge-Kutta methods are demonstrated in Figure \ref{epmfigure5}.  The Poisson integrators have shown their significant advantages in near energy conservation over long-term simulation compared with the higher order Runge-Kutta methods. We have also compared the CPU times of different explicit Poisson integrators and different Runge-Kutta methods in Table \ref{table1}. The CPU time of the 4thloba is 4 times longer than that of the 2ndEPI method.
The results show that the CPU times of the Poisson integrators are less than those of the Runge-Kutta methods.

\begin{figure}
\begin{center}
\subfigure[ ]{
\includegraphics[width=0.48\textwidth]{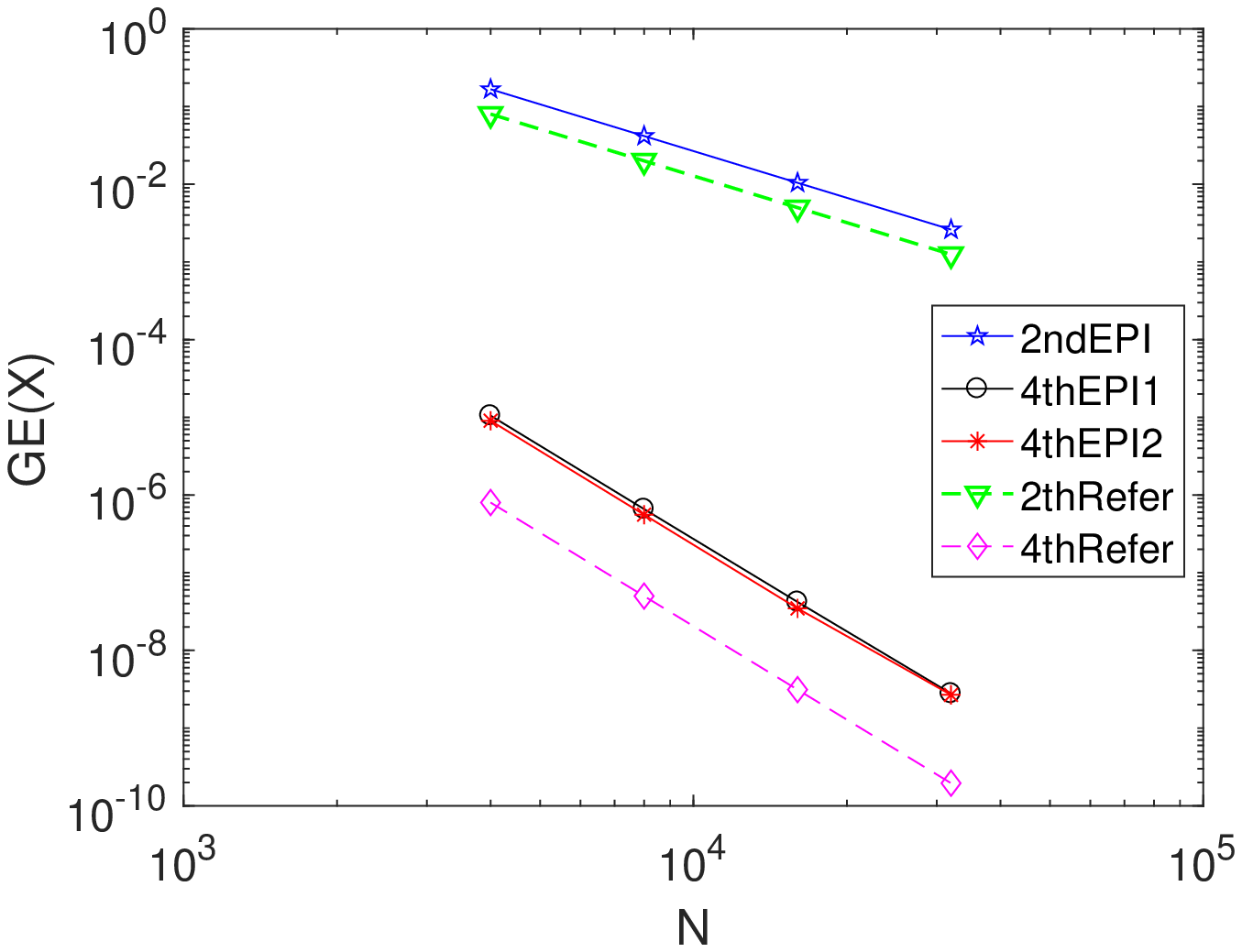}}
\subfigure[ ]{
\includegraphics[width=0.48\textwidth]{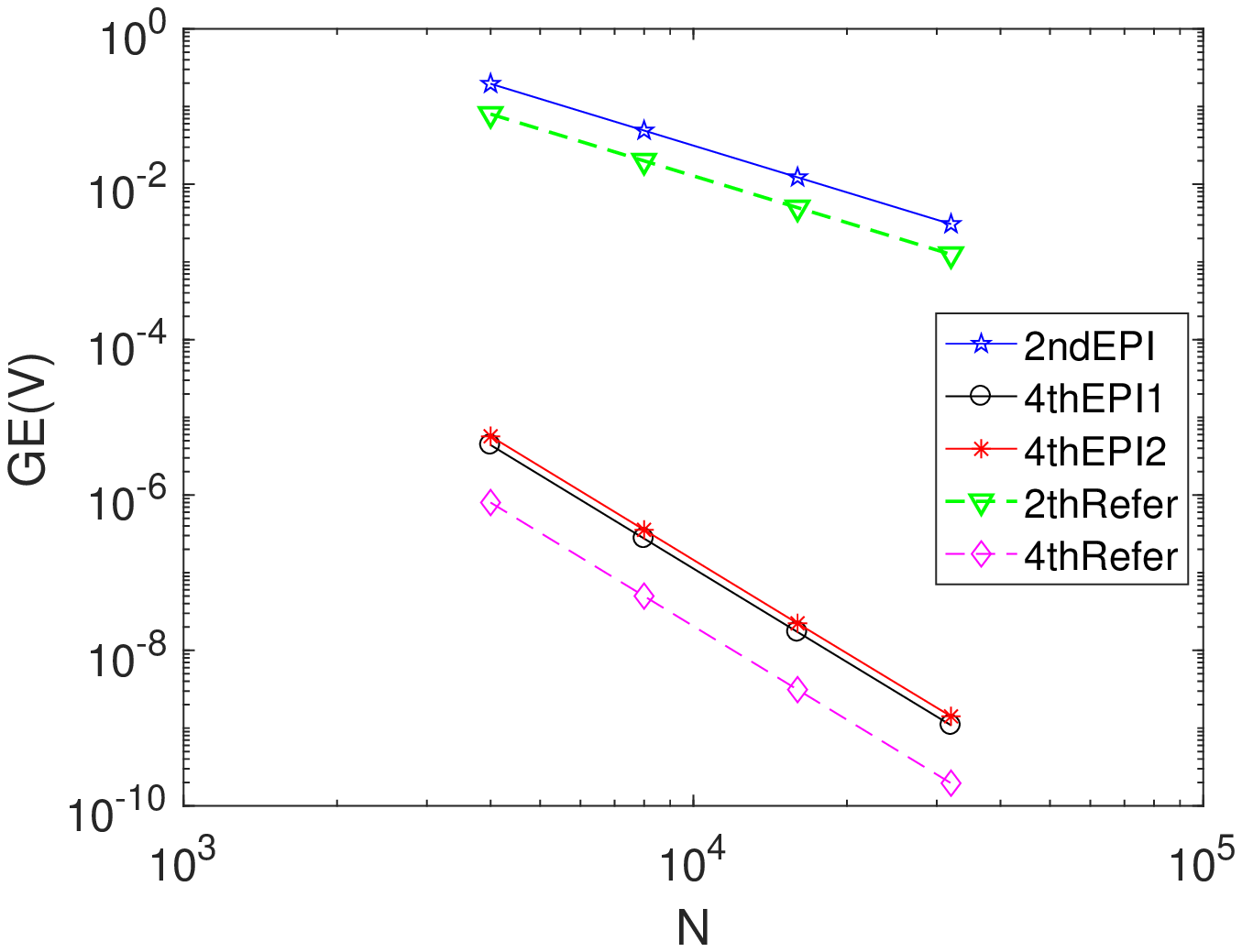}}
\end{center}
\caption{The global errors of $X$ and $V$ against
the time steps $N$ for methods 2ndEPI, 4thEPI1 and 4thEPI2 under different stepsize $h=\pi/20/2^i (i=1,2,3,4)$ in Example 2 of the charged particle system. Here the final time $T=100\pi$. $GE(X)=max_{1\le i\le N} \parallel X_i \parallel_2$. Dashed lines are the
reference lines showing the corresponding convergence orders.
Subfigures (a) shows the global errors of the variable
$X$ while subfigures (b) shows the global errors of the variable $V$.}
\label{epmfigure4}
\end{figure}

\begin{figure}
\begin{center}
\subfigure[ ]{
\includegraphics[width=0.48\textwidth]{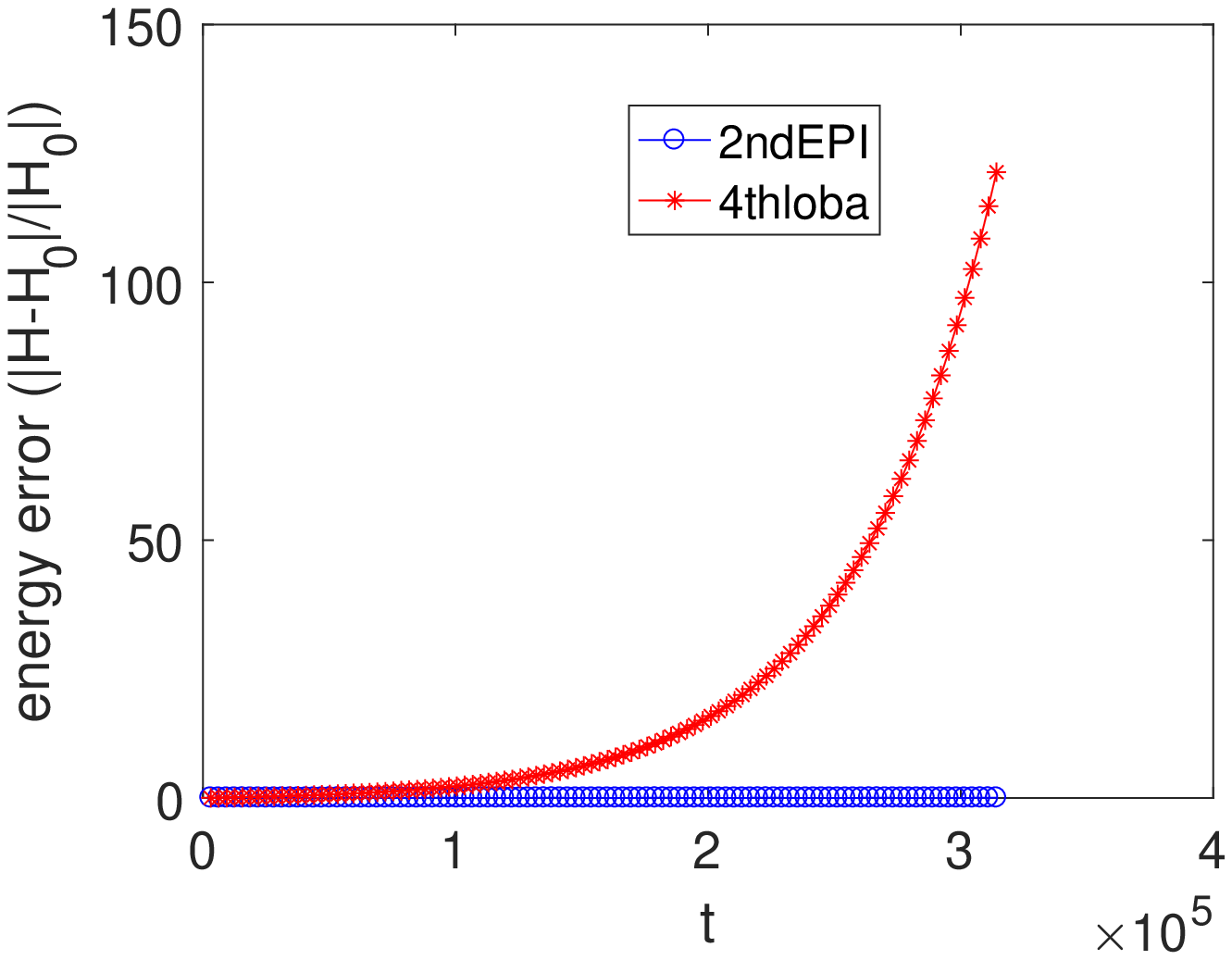}}
\subfigure[ ]{
\includegraphics[width=0.48\textwidth]{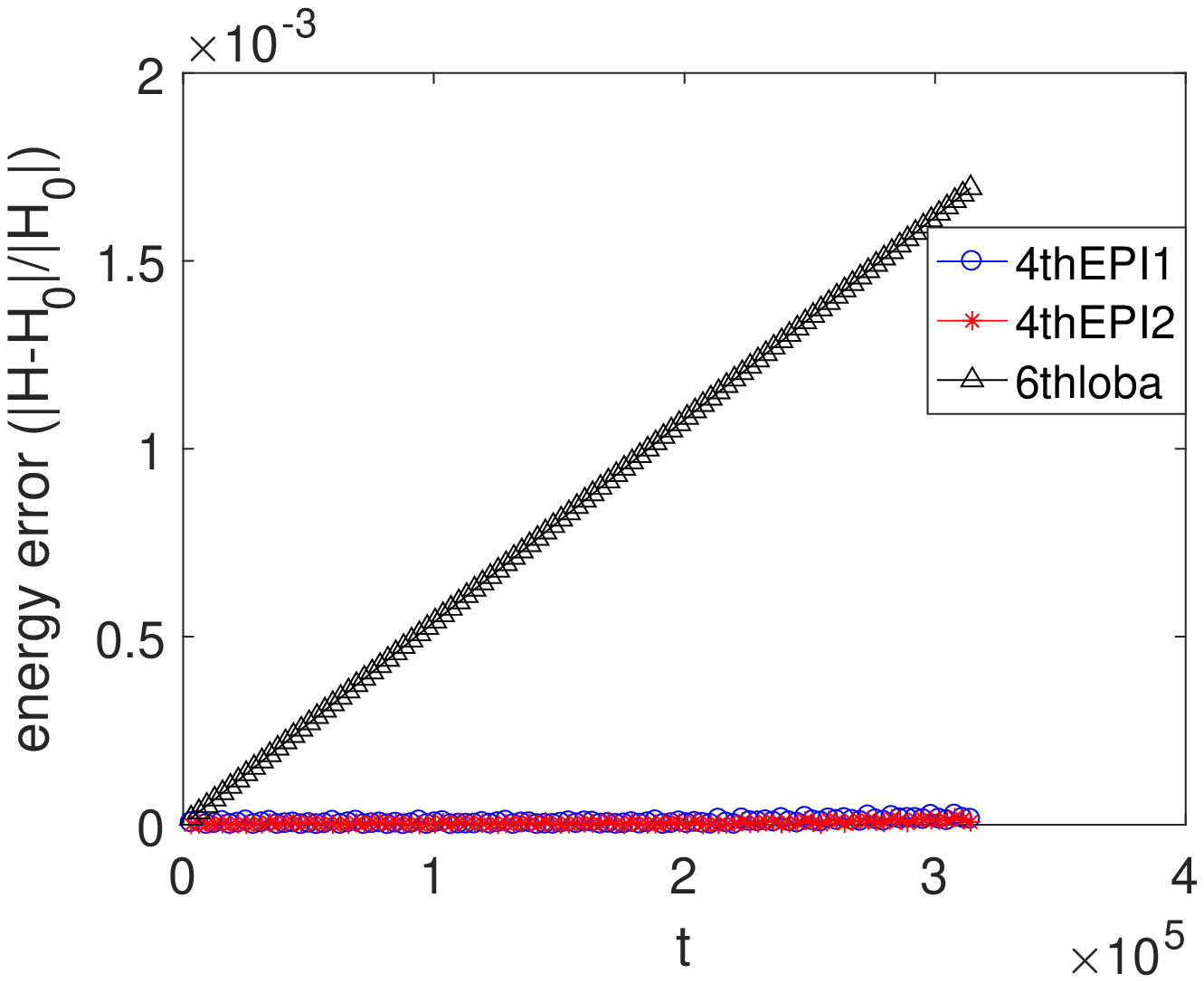}}
\end{center}
\caption{The relative energy error against $t$ for the three Poisson integrators and the two Runge-Kutta methods in Example 2 of the charged particle system. The energy
error is represented by $|H(Z_n)-H(Z_0)|/|H(Z_0)|$. The stepsize is $h=\pi/10$. Subfigure (a) displays the energy errors of the 2ndEPI method and the 4thloba method
over the time interval $[0,10^5\pi]$. Subfigure (b) displays the energy errors of the 4thEPI1 method, 4thEPI2 method and the 6thloba method over the time interval $[0,10^5\pi]$.}
\label{epmfigure5}
\end{figure}

\begin{table}[htbp]
\begin{small}
\caption{The CPU times of the five methods in Example 2 of the charged particle system. The stepsize is $h=\pi/10$ and the time interval is $[0,1000\pi]$.}
\begin{center}
\begin{tabular}{r|c|c|c|c}
\hline
2ndEPI & 4thloba & 4thEPI1 & 4thEPI2 & 6thloba\\
\hline
  0.0690 & 0.2891  &  0.3303 & 0.3984 & 0.7405 \\
\hline
\end{tabular}
\end{center}
\label{table1}
\end{small}
\end{table}

\subsection{Numerical experiments for gyrocenter system}

Here we report a few numerical experiments for two instances of the gyrocenter dynamics of one charged particle.

\textbf{Example 1:}
In the gyroocenter system, if we choose the magnetic strength $|B(X)|=c(z)$ and $b_3=0,b_1$ and $b_2$ are constants with $b_1^2+b_2^2=1$, then $$a_{23}b_1-a_{13}b_2+a_{12}b_3=\dfrac{(h_y-g_z)^2}{c(z)}-ub_1\dfrac{\partial b_2}{\partial z}+\dfrac{(f_z-h_x)^2}{c(z)}+ub_2\dfrac{\partial b_1}{\partial z}=c(z),$$thus we have
$$R(Z)=
\begin{pmatrix}
O_2  &  A  \\  -A^\top  &  O_2
\end{pmatrix}$$
with $$A=
\begin{pmatrix}
\dfrac{b_2}{c(z)} & b_1 \\-\dfrac{b_1}{c(z)} & b_2
\end{pmatrix}.$$
  By setting $b_1=b_2=\dfrac{\sqrt{2}}{2}$ and the vector potential $A(X)=(\dfrac{z^3}{3\sqrt{2}},-\dfrac{z^3}{3\sqrt{2}},0)$, then the magnetic field is $B(X)=(z^2/\sqrt{2},z^2/\sqrt{2},0)$ and $|B(X)|=c(z)=z^2$. The scalar potential is chosen to be $\varphi(X)=x^2+y^2$. We can easily verify that this matrix $R(Z)$ satisfies our requirements in Theorem \ref{theorem1}. Therefore, we can separated the gyrocenter system into two subsystems with $H_1=x^2+y^2$ and $H_2=\mu z^2+\frac{u^2}{2}$. The exact solution of the first subsystem with $H_1=x^2+y^2$ is
  \begin{equation*}
  \left\{
  \begin{split}
  x(t)&=x_0,\\
  y(t)&=y_0,\\
  z(t)&=(3\sqrt{2}(y_0-x_0)t+z_0^3)^{1/3},\\
  u(t)&=u_0-\sqrt{2}(x_0+y_0)t.
  \end{split}
  \right.
  \end{equation*}
where $x_0,y_0,z_0,u_0$ represent the initial values for $x,y,z,u$. As all the subsystems can be solved explicitly,
the explicit Poisson integrators can be constructed.

  The magnetic moment is $\mu=0.01$, the initial value is $(x_0,y_0,z_0,u_0)^\top=(30,40,60,70)^\top$. The numerical results for the five numerical methods are displayed in Figure \ref{epmfigure6}-\ref{epmfigure8}. We plot in Figure \ref{epmfigure6} the projection of the gyrocenter orbit onto the $y-u$ plane using the methods 2ndEPI, 4thEPI2, 4thloba and 6thloba.  We can see that the orbit obtained by the 4thloba method spirals outwards and is not accurate, but the lower order Poisson integrator is able to give accurate orbit.
 To illustrate the order of the Poisson integrators, we
 display in Figure \ref{epmfigure7} the global errors of the variables $X=(x,y,z)$ and $u$ and the lines clearly shows the orders of the three methods. The relative energy errors obtained by different Poisson integrators and different Runge-Kutta methods are shown in Figure \ref{epmfigure8}. The energy error of the second order Poisson integrator 2ndEPI oscillates with an amplitude of order $10^{-3}$ while the energy error of the higher order Runge-Kutta method 4thloba increases along time without bound as can be seen in Figure \ref{epmfigure8}. The energy errors of the methods 4thEPI1 and 4thEPI2 both oscillate with very small amplitudes, but that of the 6thloba method still increases linearly along time. We can also see from Figure \ref{epmfigure8} that the energy error of the 4thEPI2 method is much smaller than that of the 4thEPI1 method.

\begin{figure}
\begin{center}
\subfigure[ ]{
\includegraphics[width=0.48\textwidth]{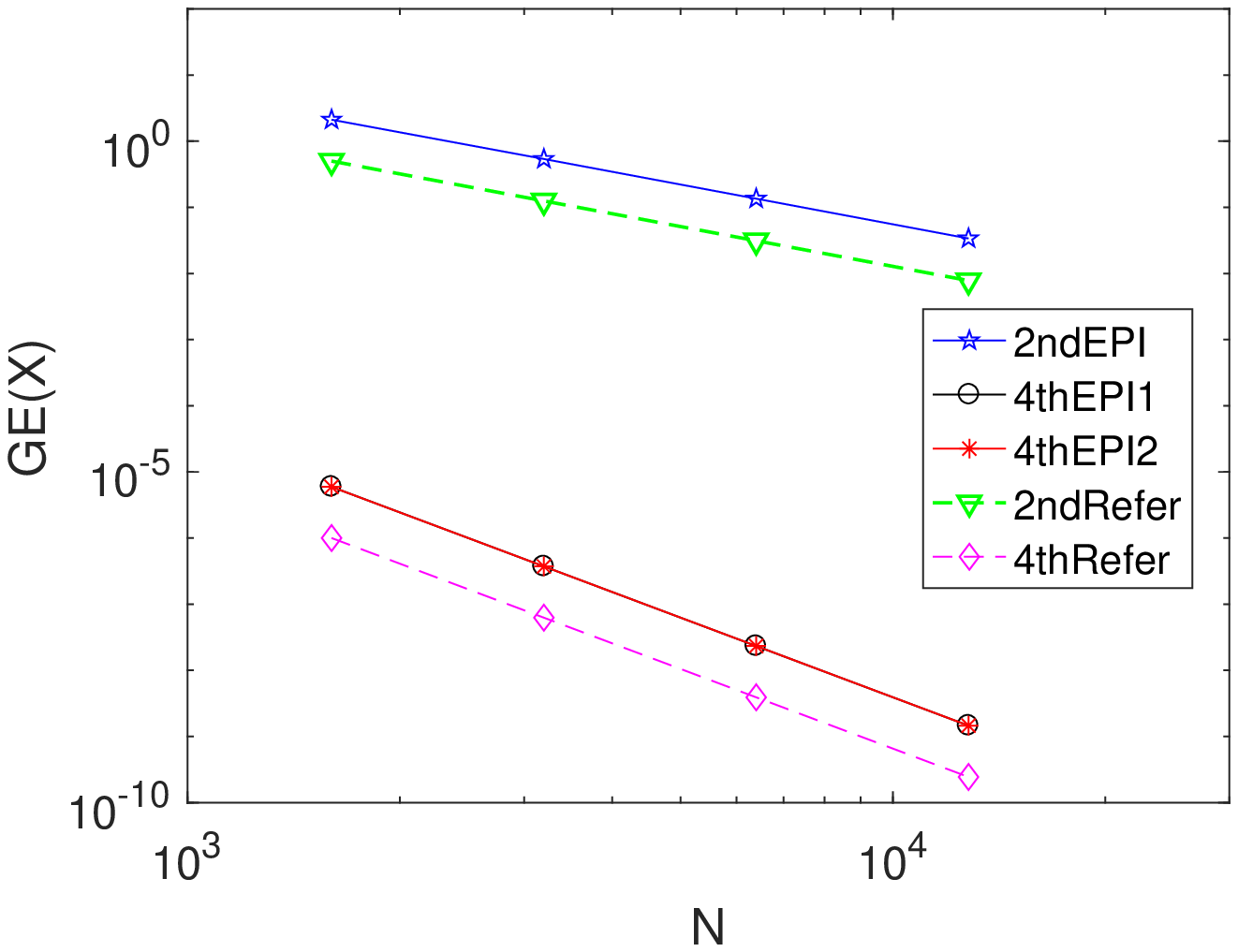}}
\subfigure[ ]{
\includegraphics[width=0.48\textwidth]{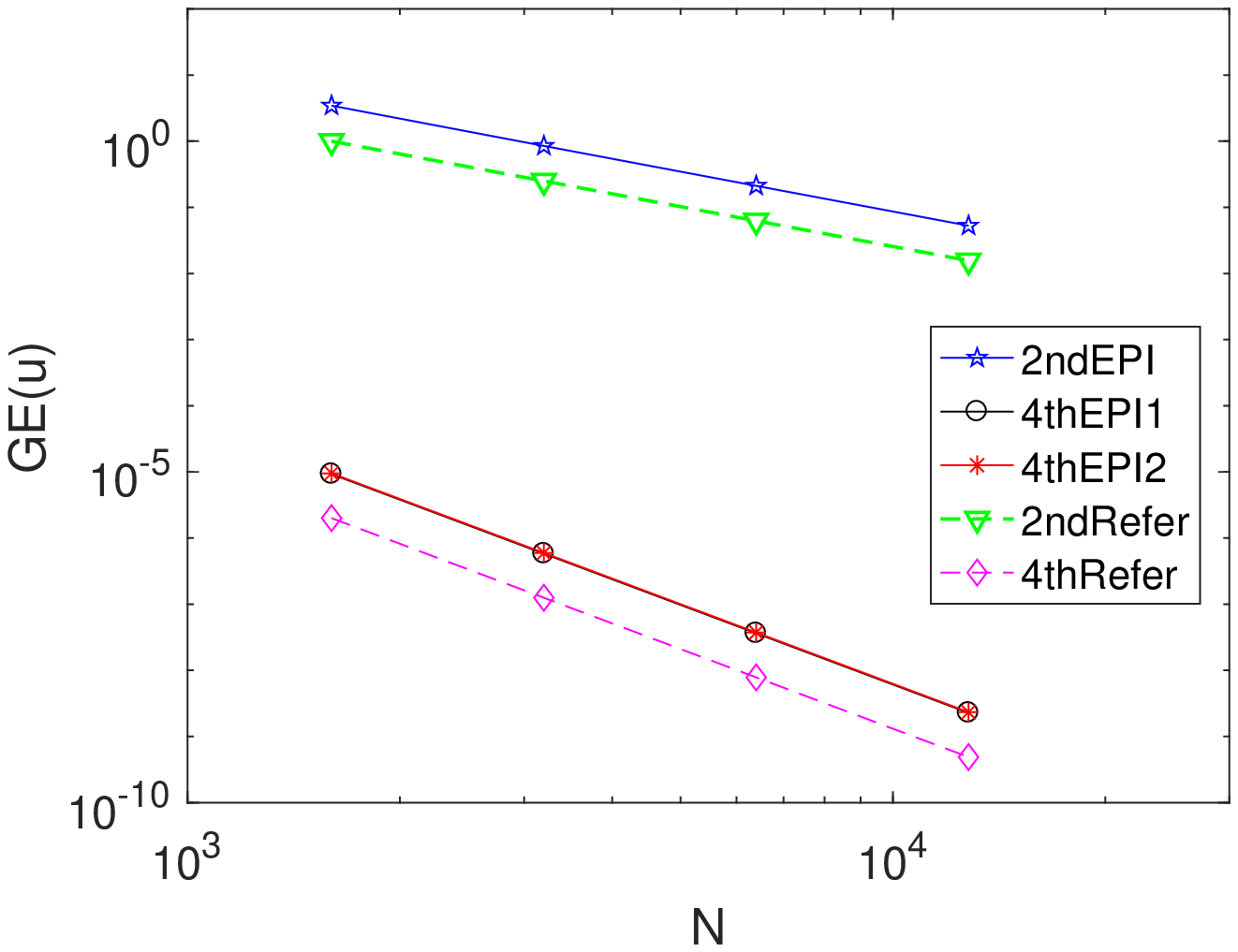}}
\end{center}
\caption{The global errors of $X$ and $u$ against
the time steps $N$ for methods 2ndEPI, 4thEPI1 and 4thEPI2 under different stepsize $h=1/2^i (i=4,5,6,7)$ in Example 1 of the gyrocenter system. Here the final time $T=100$. $GE(X)=max_{1\le i\le N} \parallel X_i \parallel_2$. Dashed lines are the
reference lines showing the corresponding convergence orders.
Subfigures (a) shows the global errors of the variable
$X$ while subfigures (b) shows the global errors of the variable $u$.}
\label{epmfigure6}
\end{figure}

\begin{figure}
\begin{center}
\subfigure[ ]{
\includegraphics[width=0.48\textwidth]{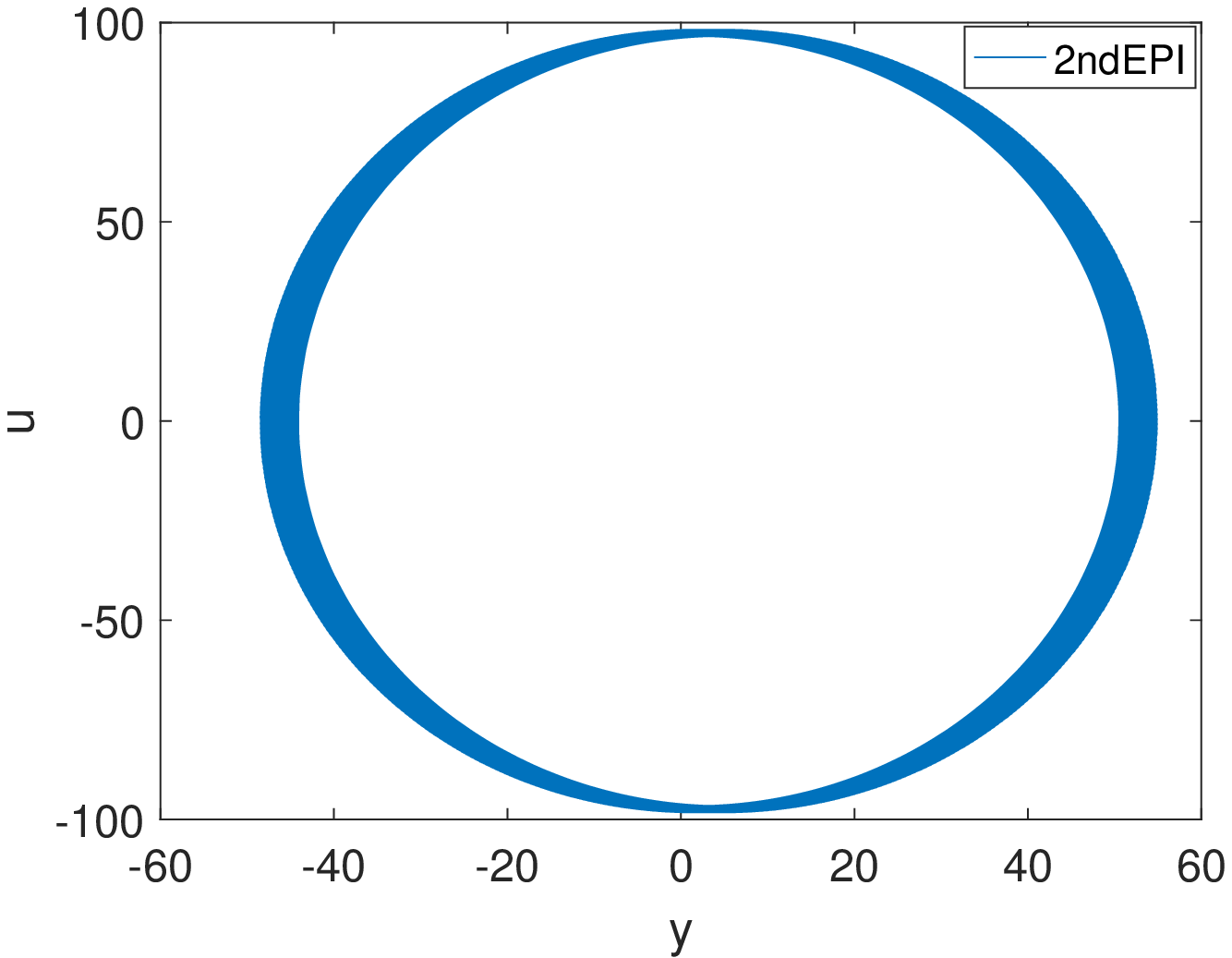}}
\subfigure[ ]{
\includegraphics[width=0.48\textwidth]{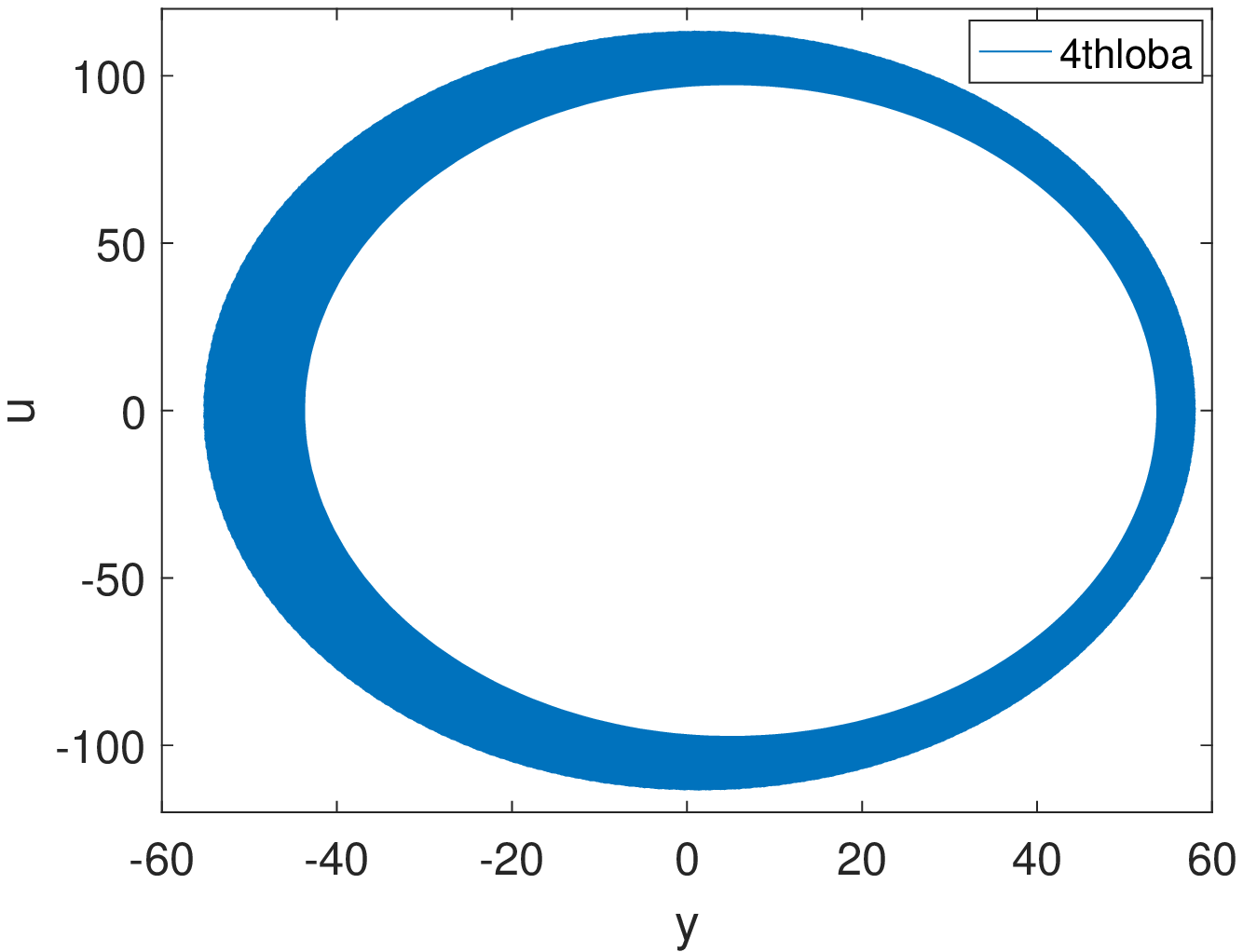}}
\subfigure[ ]{
\includegraphics[width=0.48\textwidth]{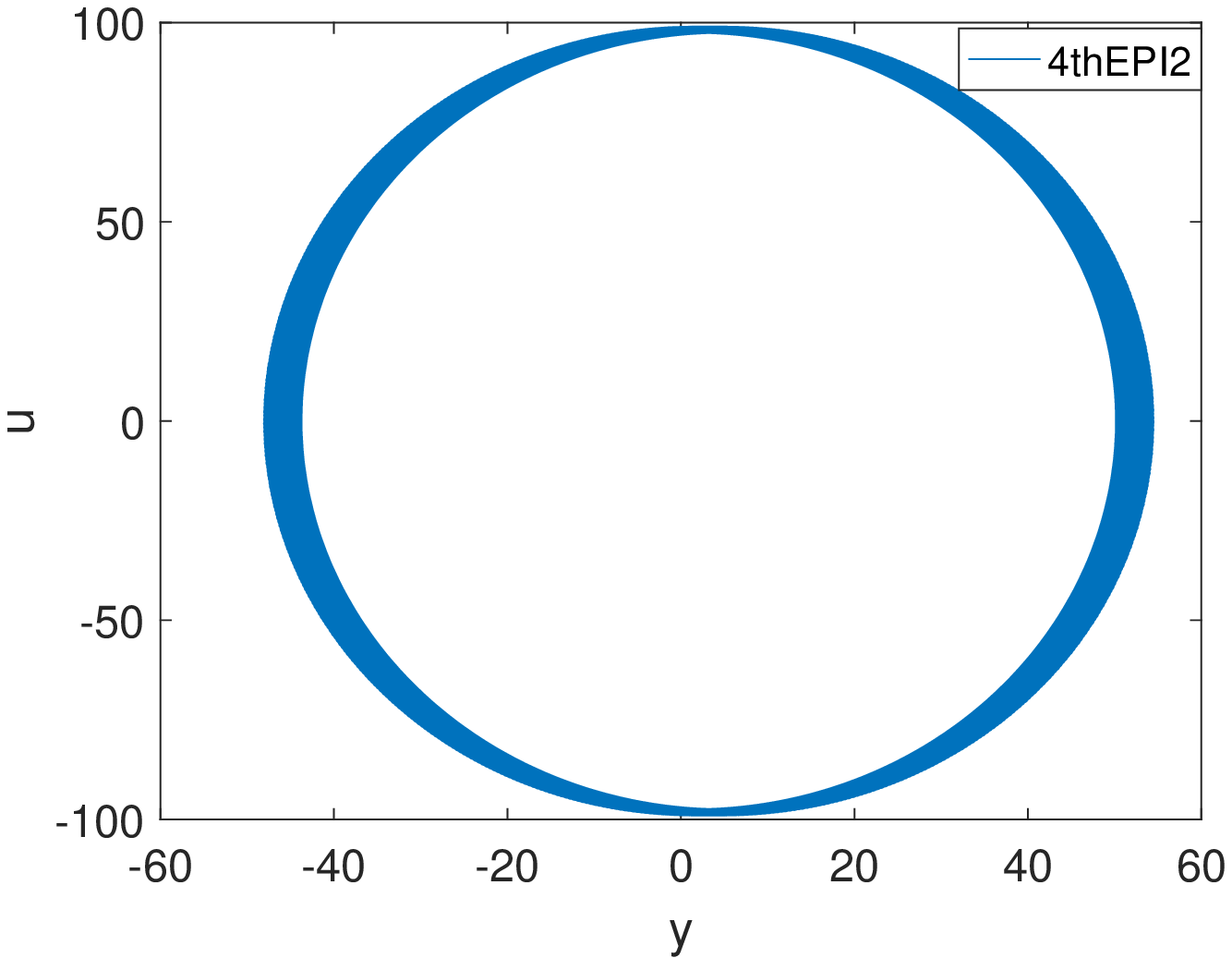}}
\subfigure[ ]{
\includegraphics[width=0.48\textwidth]{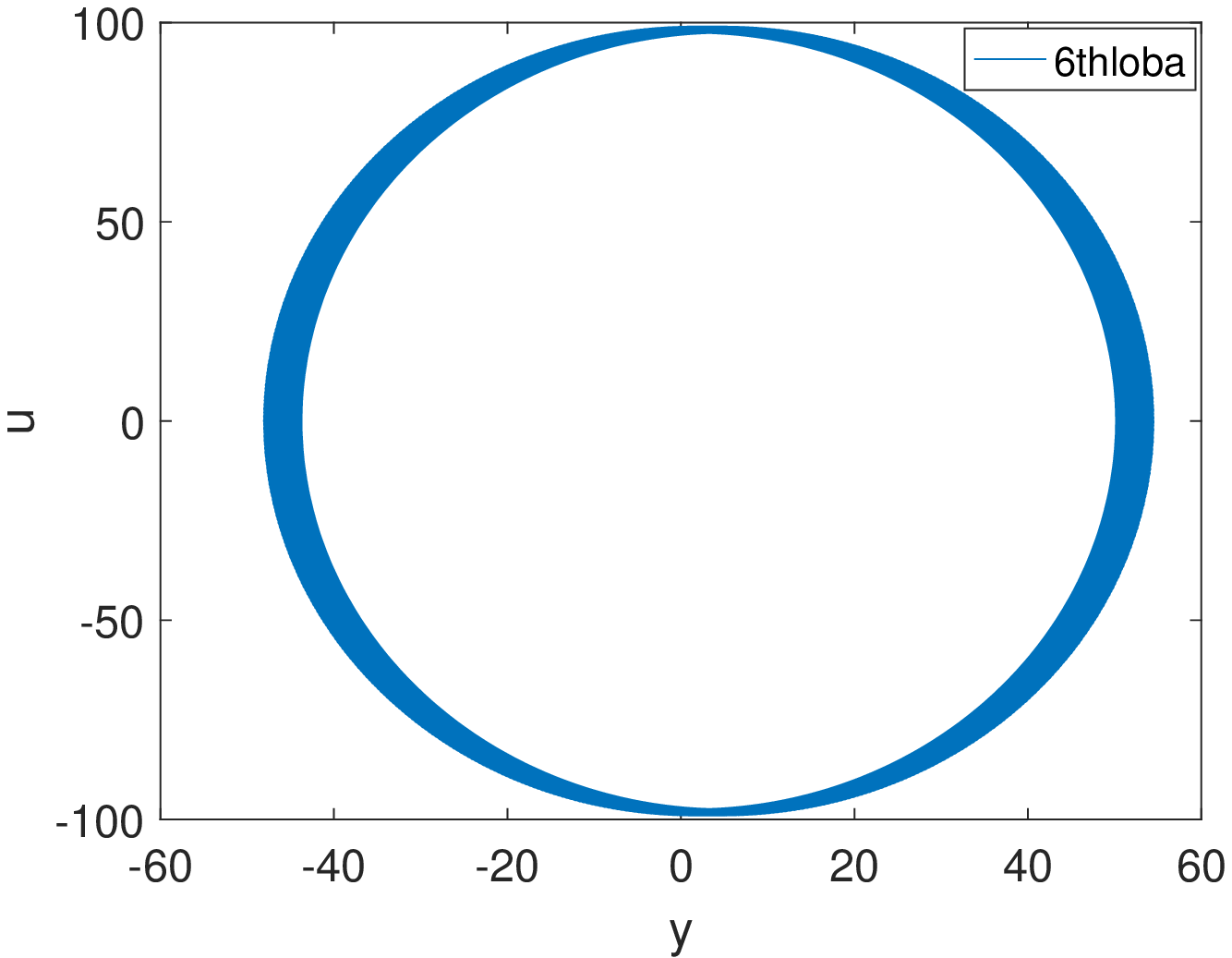}}
\end{center}
\caption{The gyrocenter orbit in $y$-$u$ plane simulated by
using the two Poisson integrators and the two Runge-Kutta methods over the interval
$[0,20000]$. The stepsize is $h=0.25$.  Subfigure (a), (b), (c) and (d) display the orbit obtained by the 2ndEPI method,
the 4thloba method, the 4thEPI2 method and
the 6thloba method, respectively.}
\label{epmfigure7}
\end{figure}

\begin{figure}
\begin{center}
\subfigure[ ]{
\includegraphics[width=0.48\textwidth]{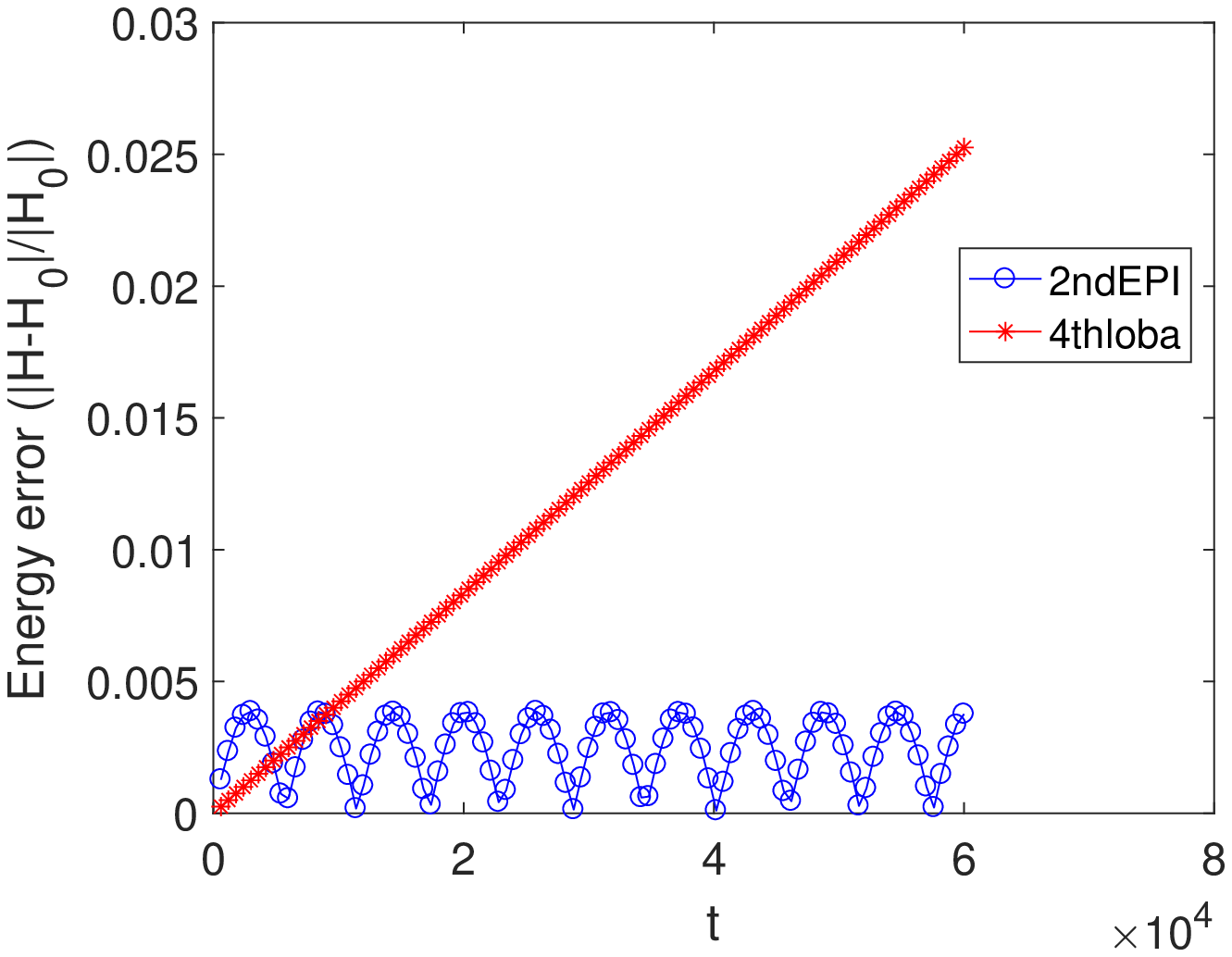}}
\subfigure[ ]{
\includegraphics[width=0.48\textwidth]{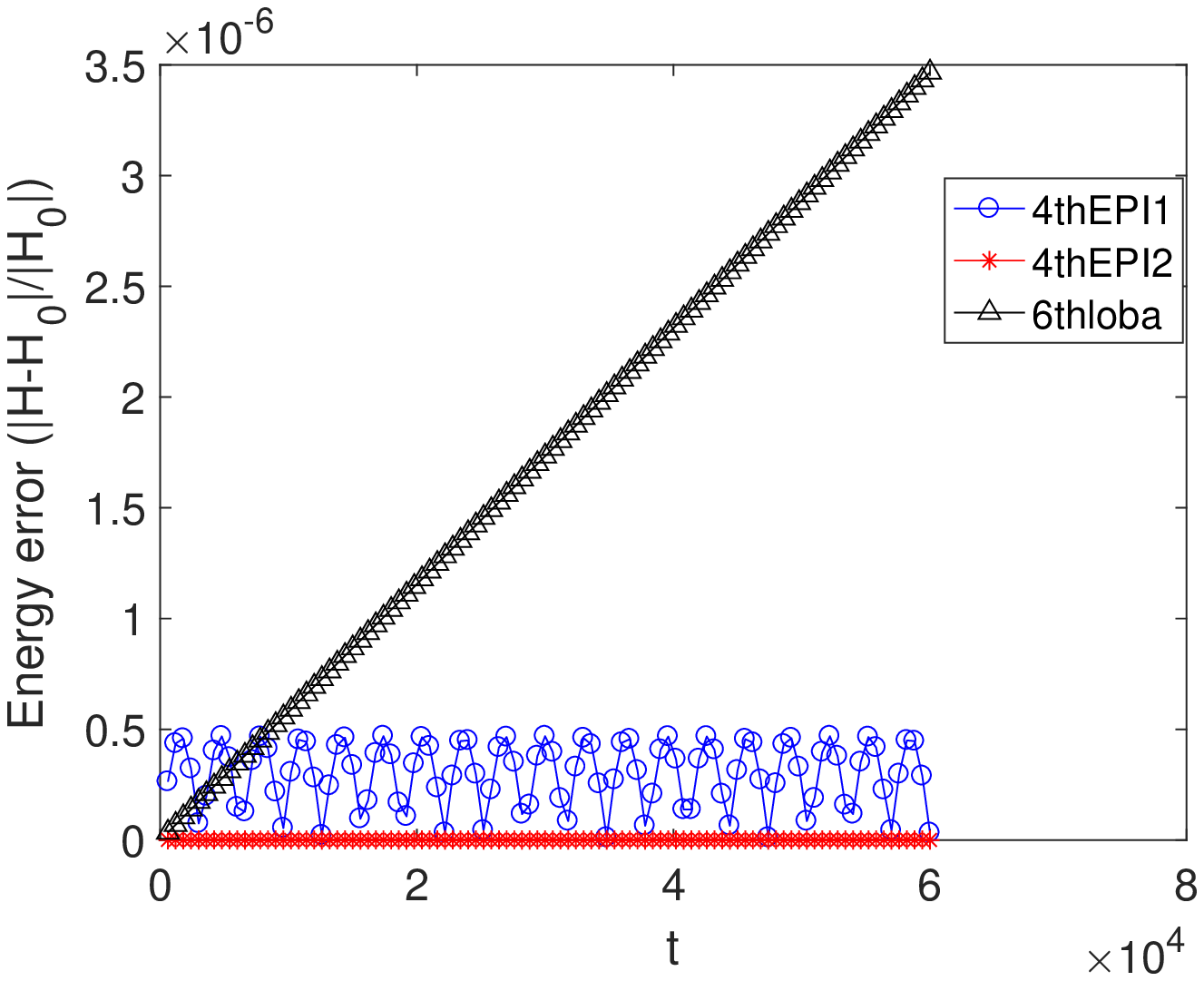}}
\end{center}
\caption{The relative energy error against $t$ for the three Poisson integrators and the two Runge-Kutta methods in Example 1 of the gyrocenter system. The energy
error is represented by $|H(Z_n)-H(Z_0)|/|H(Z_0)|$. The stepsize is $h=0.125$. Subfigure (a) displays the energy errors of the 2ndEPI method and the 4thloba method
over the time interval $[0,60000]$. Subfigure (b) displays the energy errors of the 4thEPI1 method, 4thEPI2 method and the 6thloba method over the time interval $[0,60000]$.}
\label{epmfigure8}
\end{figure}

\textbf{Example 2:}
In the gyrocenter system, we set the vector potential $A(X)=(-\frac{by^3}{3},\frac{ax^3}{3},0).$ Thus the magnetic field is $B(X)=(0,0,ax^2+by^2)$ and the magnetic strength is $|B(X)|=ax^2+by^2$. The scalar potential is set to be $\varphi(X)=2z^2$. We can easily verify under such a magnetic field, the matrix $R(Z)$ satisfies the requirements in Theorem \ref{theorem3}. The original system can be separated into four subsystems with $H_1=\mu ax^2$, $H_2=\mu by^2$, $H_3=2z^2$ and $H_4=\frac{u^2}{2}$. Here we only present the exact solution of the first subsystem with $H_1=\mu ax^2$
\begin{equation*}
\left\{
\begin{split}
&x(t)=x_{0},\\
&\frac{x_0}{2}y(t)+\frac{b}{6ax_0}y(t)^3=\mu t+\frac{x_0}{2}y_0+\frac{by_0^3}{6ax_0},\\
&z(t)=z_0,\\
&u(t)=u_{0}.
\end{split}
\right.
\end{equation*}

The magnetic moment is chosen as $\mu=0.001$, and the initial value is $(x_0,y_0,z_0,u_0)^\top=(30,20,40,50)^\top$. The global errors of the variables $X=(x,y,z)$ and $u$ of the three Poisson integrators 2ndEPI, 4thEPI1 and 4thEPI2 are plotted in Figure \ref{epmfigure4}.  The lines in Figure \ref{epmfigure4} show the orders of these methods. It can also be seen from Figure \ref{epmfigure4} that the global error of the method 4thEPI2 for the variable $u$ is smaller than that of the method 4thEPI1.
The relative energy errors of the five methods are shown in Figure \ref{epmfigure10}. The Poisson integrators have significant superiorities in preserving the energy over long time compared with the higher order Runge-Kutta methods as can be seen in Figure \ref{epmfigure10}. We have also compared the computational costs of the five methods in Table \ref{table2}. The results show that the computational costs of the two Runge-Kutta methods are more than 12 times higher than those of the Poisson integrators. The Poisson integrators has shown their accuracy, efficiency and long-term energy conservation in simulating the gyrocenter system compared with the higher order Runge-Kutta methods.

\begin{figure}
\begin{center}
\subfigure[ ]{
\includegraphics[width=0.48\textwidth]{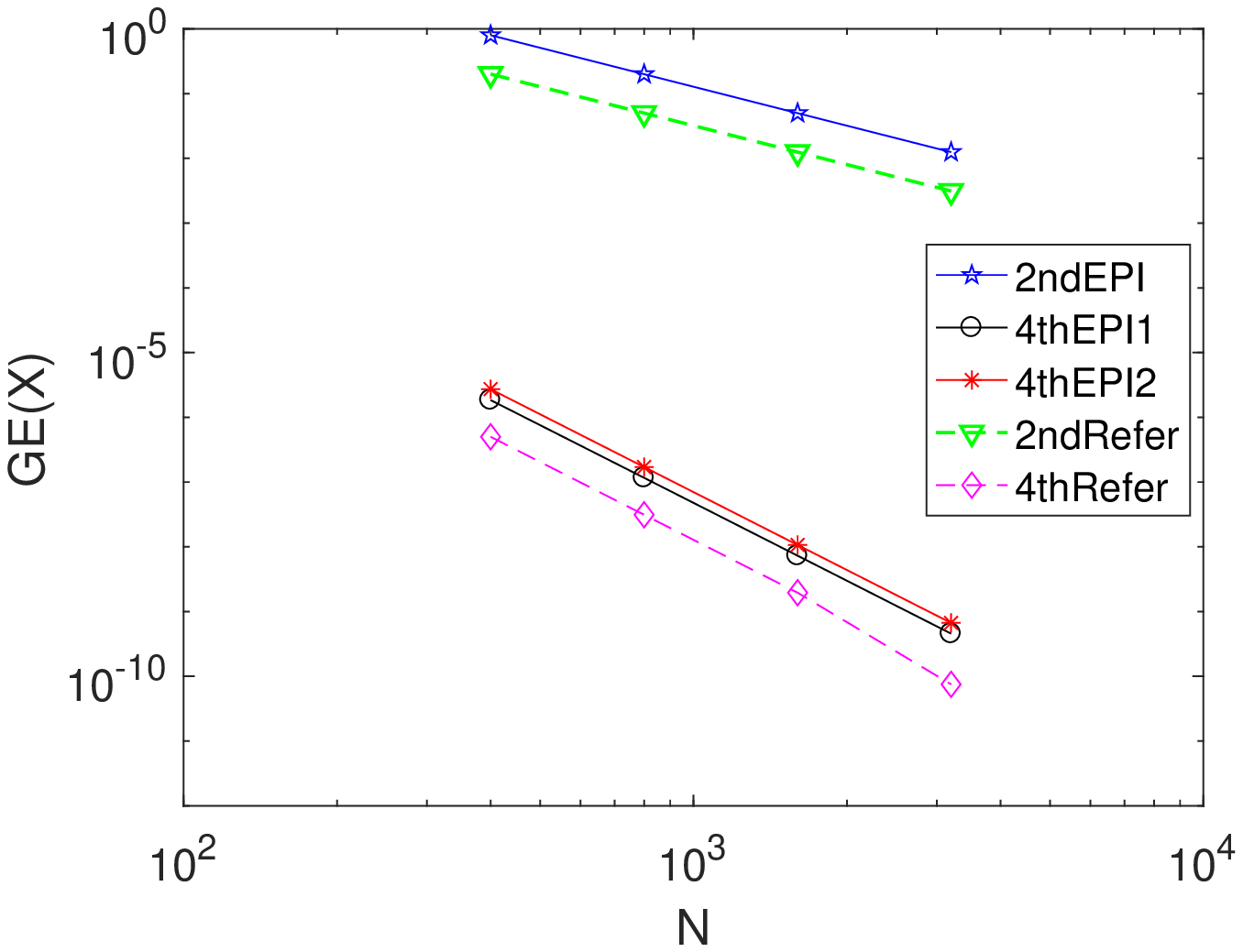}}
\subfigure[ ]{
\includegraphics[width=0.48\textwidth]{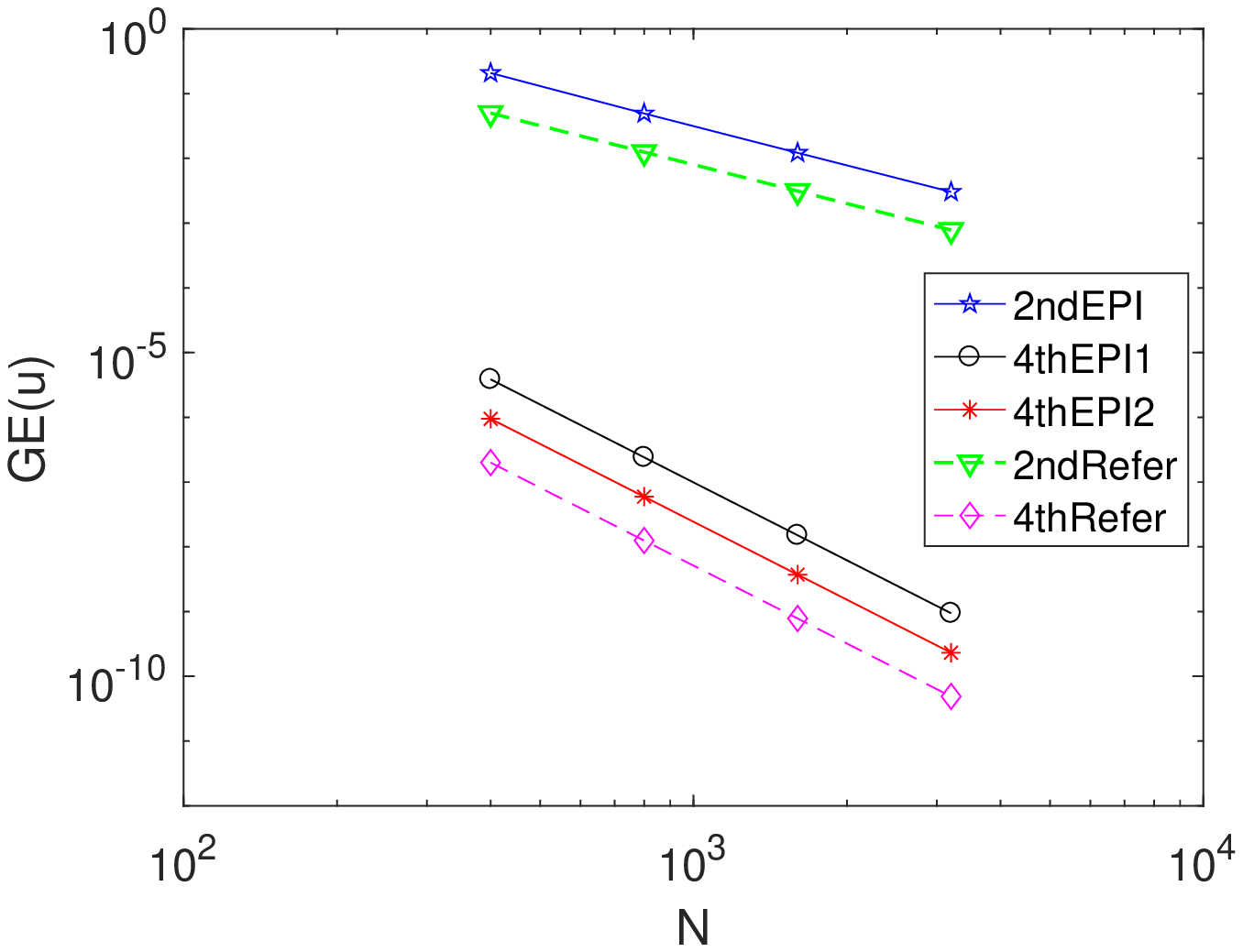}}
\end{center}
\caption{The global errors of $X$ and $u$ against
the time steps $N$ for methods 2ndEPI, 4thEPI1 and 4thEPI2 under different stepsize $h=0.1/2^i (i=1,2,3,4)$ in Example 2 of the gyrocenter system. Here the final time $T=20$. $GE(X)=max_{1\le i\le N} \parallel X_i \parallel_2$. Dashed lines are the
reference lines showing the corresponding convergence orders.
Subfigures (a) shows the global errors of the variable
$X$ while subfigures (b) shows the global errors of the variable $u$.}
\label{epmfigure9}
\end{figure}

\begin{figure}
\begin{center}
\subfigure[ ]{
\includegraphics[width=0.48\textwidth]{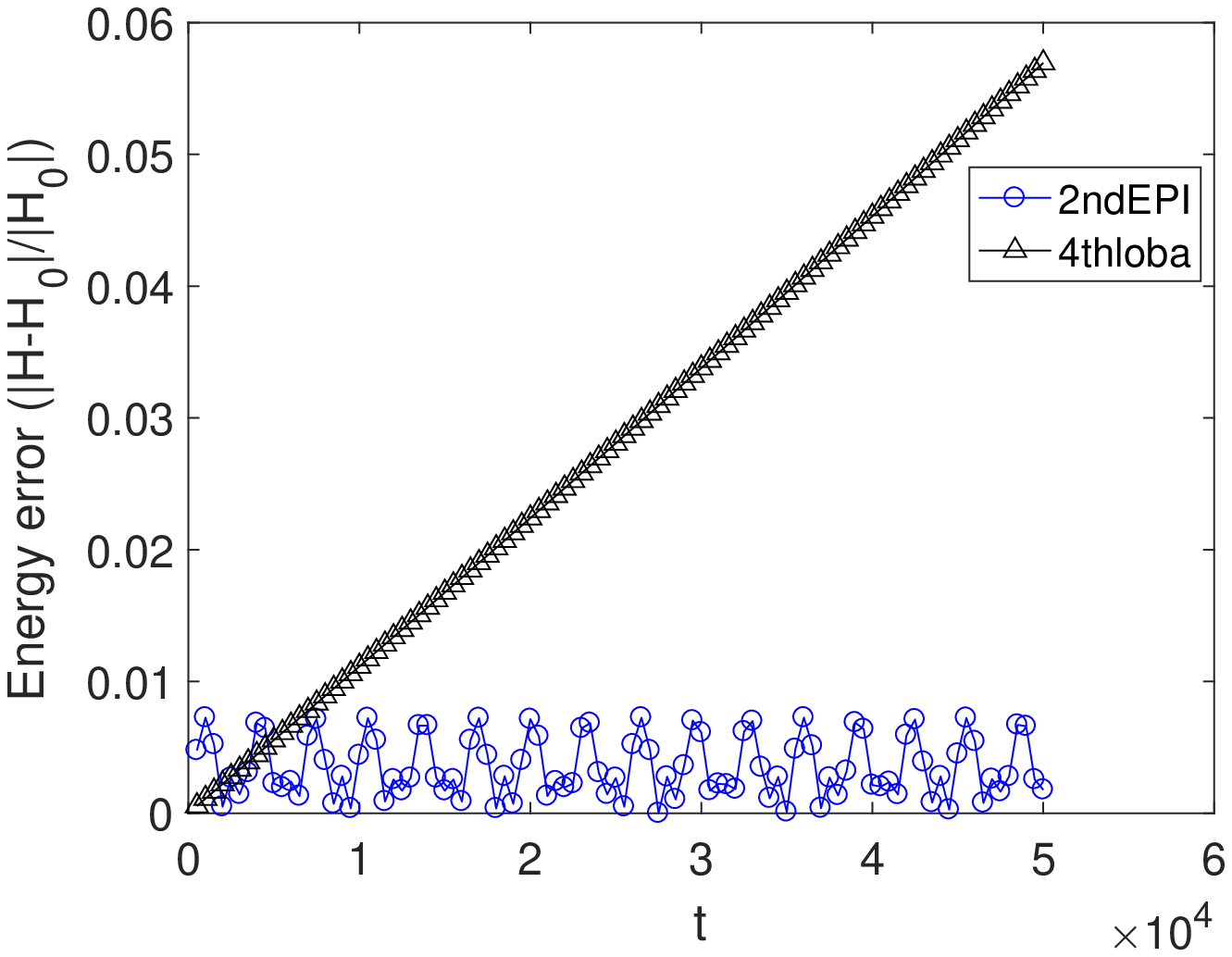}}
\subfigure[ ]{
\includegraphics[width=0.48\textwidth]{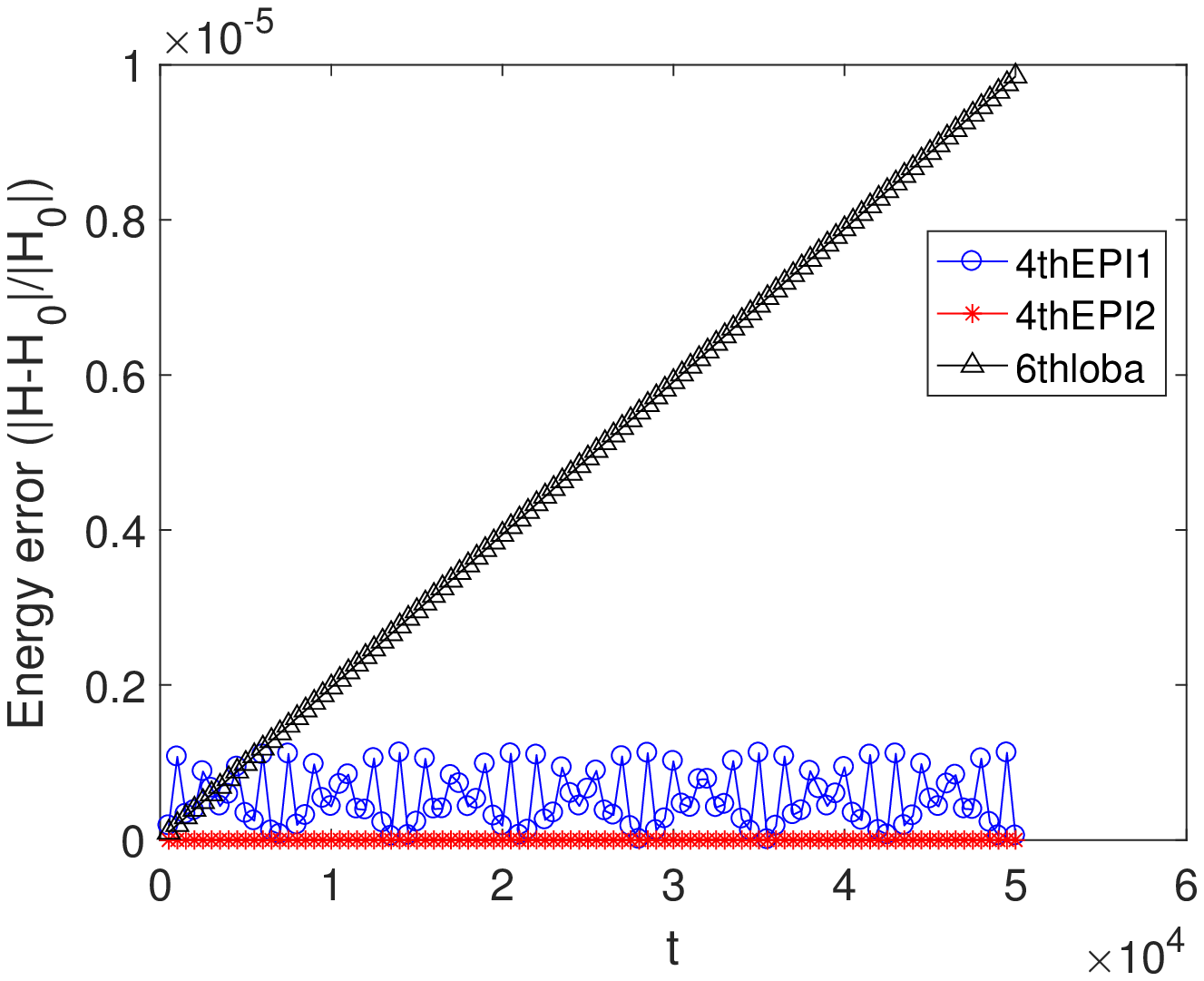}}
\end{center}
\caption{The relative energy error against $t$ for the three Poisson integrators and the two Runge-Kutta methods in Example 1 of the gyrocenter system. The energy
error is represented by $|H(Z_n)-H(Z_0)|/|H(Z_0)|$. The stepsize is $h=0.1$. Subfigure (a) displays the energy errors of the 2ndEPI method and the 4thloba method
over the time interval $[0,50000]$. Subfigure (b) displays the energy errors of the 4thEPI1 method, 4thEPI2 method and the 6thloba method over the time interval $[0,50000]$.}
\label{epmfigure10}
\end{figure}

\begin{table}[htbp]
\begin{small}
\caption{The CPU times of the five methods in Example 2 of the gyrocenter system. The stepsize is $h=0.1$ and the time interval is $[0,200]$.}
\begin{center}
\begin{tabular}{r|c|c|c|c}
\hline
2ndEPI & 4thloba & 4thEPI1 & 4thEPI2 & 6thloba\\
\hline
 0.7261 & 16.9257  &  3.2026 & 3.8296 & 46.1720 \\
\hline
\end{tabular}
\end{center}
\label{table2}
\end{small}
\end{table}

\section{Conclusion}

Poisson integrators for Poisson systems with separable Hamiltonian have been constructed via splitting method. We have separated the Poisson systems in three ways and analyzed three situations where the Poisson integrators can be constructed. The second order and the fourth order Poisson integrators have been constructed by composing the first order Poisson integrator.

We have compared the second order and the fourth order Poisson integrators with the Runge-Kutta methods whose orders are higher than them to show their superiority in simulating two Poisson systems: the charged particle system and the gyrocenter system. Numerical experiments show that the constructed explicit Poisson integrators have significant advantages in preserving the phase orbit and long-term energy conservation compared with the higher order Runge-Kutta methods. The explicit Poisson integrators take less computational costs than the Runge-Kutta methods, as well.

\section*{Acknowledgments}

This research is supported by
the National Natural Science Foundation of China (Grant Nos. 11901564 and 12171466).

\bibliographystyle{abbrv}
\bibliography{main}

\begin{thebibliography}{99}

\bibitem{Arnold} V. Arnold, Mathematical Methods of Classical Mechanics, Springer, New York, 1978.
\bibitem{Blanes} S. Blanes, P.C. Moan, Practical symplectic partitioned Runge-Kutta and Runge-Kutta-Nystr\"{o}m methods, J. Comput. Appl. Math. 142 (2002) 313-330.
\bibitem{Butcher1} J.C. Butcher, Implicit Runge-Kutta Processes. Math. Comput., 18 (1964) 50-64.
\bibitem{Butcher2} J.C. Butcher, Integration processes based on Radau quadrature formulas. Math.
Comput., 18 (1964) 233-244.
\bibitem{Channell} P. Channell, C. Scovel, Symplectic integration of Hamiltonian systems, Nonlinearity 3 (1990) 231-259.
\bibitem{Channell91} P.J. Channell, J.C. Scovel, Integrators for Lie-Poisson dynamical systems, Phys. D. 50 (1991) 80-88.
\bibitem{Faou} E. Faou, Ch. Lubich, A Poisson integrator for Gaussian wavepacket dynamics, Comput. Visual Sci. 9 (2006) 45-55.
\bibitem{Feng1984} K. Feng, in: K. Feng (Ed.), Proceedings of 1984 Beijing Symposium on Differential Geometry and Differential Equations, Science Press, Beijing, 1985,
pp. 42-58.
\bibitem{FPS} K. Feng, Collected Works of Feng Kang (II), National Defence Industry Press, Beijing, 1995.
\bibitem{Forest} E. Forest, R. Ruth, Fourth-order symplectic integration, Physica D. 43 (1990) 105-117.
\bibitem{GeZ} Z. Ge, J. Marsden, Lie-Poisson Hamilton-Jacobi theory and Lie-Poisson integrators, Phys. Lett. A. 133 (1988) 134-139.
\bibitem{GNT} E. Hairer, Ch. Lubich, and G. Wanner, Geometric Numerical Integration: Structure-Preserving Algorithms for Ordinary Differential Equations, Springer, 2002.
\bibitem{Zhou} Y. He, Z. Zhou, Y. Sun, J. Liu, H. Qin, Explicit K-symplectic algorithms for charged particle dynamics, Phys. Lett. A. 381 (2017) 568-573.
\bibitem{LiYZ} Y. Li, Y. He, Y. Sun, J. Niesen, H. Qin, J. Liu, Solving the Vlasov-Maxwell equations using Hamiltonian splitting, J. Comput. Phys. 396 (2019) 381-399.
\bibitem{LiT} T. Li, B. Wang, Efficient energy-preserving methods for charged-particle dynamics, Appl. Math. Comput. 361 (2019) 703-714.
\bibitem{Littlejohn} R. Littlejohn, A guiding center Hamiltonian: A new approach, J. Math. Phys. 20 (1979) 2445-2458.
\bibitem{Lie} S. Lie, Zur Theorie der Transformationsgruppen, Christ. Forh. Aar. 1888, Nr. 13, 6 pages, Christiania, Gesammelte Abh., 1888 (5) 553-557.
\bibitem{McLachlan} R.I. McLachlan, On the numerical integration of ordinary differential equations by symmetric
composition methods, SIAM J. Sci. Comput. 16 (1995) 151-168.
\bibitem{Morrison2} P. Morrison, The Maxwell-Vlasov equtions as a continuous Hamiltonian system, Phys. Lett, 80A (1980) 383-386.
\bibitem{Morrison1} P. Morrison, J. Green, Non-canonical Hamiltonian desentity formulation of hydrodynamics and idel magnetohydrdynamics, Phys. Rev. Letters, 45 (1980) 790-794.
\bibitem{Qin} H. Qin, X. Guan, W. Tang, Variational symplectic algorithm for guiding center dynamics and its application in tokamak geometry, Phys. Plasmas, 16 (2009) 042510.
\bibitem{Sanzserna} J.M. Sanz-Serna, Runge-Kutta schemes for Hamiltonian systems, BIT Numer. Math. 28 (1988) 877-883.
\bibitem{SSC94} J. M. Sanz-Serna and M. P. Calvo, Numerical Hamiltonian Problems, Chapman and Hall, London, 1994.
\bibitem{Strang} G. Strang, On the construction and comparison of difference schemes, SIAMJ. Numer. Anal. 5 (1968) 507-517.
\bibitem{Suris} Y. Suris, Integrable discretizations for lattice systems: local equations of motion and their Hamiltonian properties, Rev. Math. Phys, 11 (1999) 727-822.
\bibitem{Tang1} Y. Tang, V. P\'erez-Garc\'ia, L. V\'azquez, Symplectic Methods for the Ablowitz-Ladik Discrete Nonlinear Schr\"{o}dinger Equation, Appl. Math. Comput. 82 (1997) 17-38.
\bibitem{Touma} J. Touma, J. Wisdom, Lie-Poisson integrators for rigid body dynamics in the solar system, Astron. J. 107 (1994) 1189-1202.
\bibitem{Wisdom} J. Wisdom, M. Holman, Symplectic maps for the n-body problem, Astron. J. 104 (1992) 1528-1538.
\bibitem{Zhang} R. Zhang, J. Liu, Y. Tang, H. Qin, J. Xiao, and B. Zhu, Canonicalization and symplectic simulation of the gyrocenter dynamics in time-independent magnetic fields, Phys. Plasma. 21 (2014) 032504.
\bibitem{Zhang2} R. Zhang, H. Qin, Y. Tang, J. Liu, Y. He, J. Xiao, Explicit symplectic algorithms based on generating functions for charged particle dynamics, Phys. Rev. E. 2016 (94) 013205.
\bibitem{Zhu} B. Zhu, R. Zhang, Y. Tang, X. Tu and Y. Zhao, Splitting K-symplectic methods for non-canonical separable Hamiltonian problems, J. Comp. Phys., 322(2016) 387-399.
\bibitem{Zhu2} B.Zhu, Y. Tang, J. Liu, Energy-preserving methods for guiding center system based on averaged vector field, Phys. Plasmas, 29 (2022) 032501.


\end{thebibliography}

\end{document}